\documentclass[12pt,reqno,oneside]{amsart}
\usepackage{amsmath}
\usepackage{amssymb}
\usepackage{amsthm}
\usepackage{graphics}
\usepackage{eucal}
\usepackage{bm}
\usepackage{url}
\usepackage{hyperref}
\usepackage[capitalize,nameinlink,noabbrev]{cleveref}
\usepackage{mathrsfs}
\usepackage[normalem]{ulem}
\usepackage{color}
\usepackage[dvipsnames]{xcolor}
\newcommand{\red}[1]{{\color{red}#1}}

\usepackage{mathtools}
\usepackage{geometry}
\newtheorem{theorem}{Theorem}[section]
\newtheorem{definition}{Definition}[section]
\newtheorem{proposition}{Proposition}[section]
\newtheorem{lemma}{Lemma}[section]
\newtheorem{corollary}{Corollary}[section]

\newtheorem{remark}{Remark}[section]

\numberwithin{equation}{section}

\newcommand{\be}{\mathbf{e}}

\newcommand{\dist}{\operatorname{dist}}

\newcommand{\supp}{\operatorname{supp}}

\newcommand{\bt}{\mathbf{t}}
\newcommand{\br}{\mathbf{r}}
\newcommand{\bq}{\mathbf{q}}
\newcommand{\bx}{\mathbf{x}}

\newcommand{\cL}{\mathcal{L}}
\newcommand{\cM}{\mathcal{M}}

\newcommand{\bw}{\mathbf{w}}
\newcommand{\f}{\mathbf{f}}

\newcommand{\bz}{\mathbf{z}}

\newcommand{\by}{\mathbf{y}}

\newcommand{\Q}{\mathbb{Q}}

\newcommand{\Z}{\mathbb{Z}}
\newcommand{\R}{\mathbb{R}}
\newcommand{\N}{\mathbb{N}}

\def\red{\textcolor{red}}
\newcommand{\munu}{\mu}

\usepackage{mathtools}
\usepackage{tikz-cd}
\usepackage{graphicx}
\begin{document}

\title{Good functions, measures, and the Kleinbock-Tomanov conjecture}

\author{Victor Beresnevich}
\address{\textbf{Victor Beresnevich}\\
Department of Mathematics, University of York, Heslington, York, UK}
\email{victor.beresnevich@york.ac.uk }

\author{Shreyasi Datta}
\address{\textbf{Shreyasi Datta} \\
Department of Mathematics, University of York, Heslington, York, UK}
\email{shreyasi1992datta@gmail.com, shreyasi.datta@york.ac.uk}

\author{Anish Ghosh}
\address{\textbf{Anish Ghosh} \\
School of Mathematics,
Tata Institute of Fundamental Research, Mumbai, India 400005}
\email{ghosh@math.tifr.res.in}
\date{}

\thanks{AG was supported by the Government of India, Department of Science and Technology, Swarnajayanti fellowship DSTSJFMSA-01/2016-17 and a grant from the Infosys Foundation. AG acknowledges support of the Department of Atomic Energy, Government of India [under project 12 - R\&D - TFR - 5.01 - 0500].}

\date{}

\begin{abstract}
In this paper we prove a conjecture of Kleinbock and Tomanov \cite[Conjecture~FP]{KT} on Diophantine properties of a large class of fractal measures on $\Q_p^n$. More generally, we establish the $p$-adic analogues of the influential results of Kleinbock, Lindenstrauss and Weiss \cite{KLW} on Diophantine properties of friendly measures.  We further prove the $p$-adic analogue of one of the main results in \cite{Kleinbock-exponent} concerning Diophantine inheritance of affine subspaces, which answers a question of Kleinbock. One of the key ingredients in the proofs of \cite{KLW} is a result on $(C, \alpha)$-good functions whose proof crucially uses the Mean Value Theorem. Our main technical innovation is an alternative approach to establishing that certain functions are $(C, \alpha)$-good in the $p$-adic setting. We believe this result will be of independent interest.
\end{abstract}
\maketitle

\tableofcontents

\section{Introduction}

The problems considered in this paper go back to the landmark work of Kleinbock and Margulis \cite{KM} from 1998 which settled the Baker-Sprind\v{z}huk conjecture and ushered in major sustained progress in the area of Diophantine approximation on manifolds. Several years later Kleinbock, Lindenstrauss and Weiss \cite{KLW} transformed the area by introducing the general framework of Diophantine properties of measures and developing the relevant Diophantine theory for a large natural class of measures on $\R^n$ which they named \emph{friendly measures}. Specifically, they proved that almost every point in the support of a friendly measure is not very well approximable by rationals. Since the class of friendly measures includes smooth measures on nondegenerate manifolds, \cite{KLW} generalizes the results of Kleinbock and Margulis \cite{KM}. Moreover, there are many other interesting examples of friendly measures such as measures supported on self-similar sets satisfying the open set condition \cite{KLW, U}, Patterson-Sullivan measures associated with convex cocompact Kleinian groups \cite{SU06} and Patterson-Sullivan measures associated with some geometrically finite groups \cite{DFSU}.

In another direction, Kleinbock and Tomanov \cite{KT} proved $p$-adic, and more generally $S$-arithmetic, versions of many of the results from \cite{KLW}. In particular, they established the $S$-arithmetic analogue of the Baker-Sprind\v{z}huk conjecture. In the same paper, Kleinbock and Tomanov conjectured that $p$-adic analogues of the main results in \cite{KLW} should hold, see Conjecture~FP in \cite{KT}, where `FP' stands for `Friendliness of Pushforwards'. Conjecture~FP requires establishing certain properties of pushforwards of measures subject to natural constrains, and via the theory developed in \cite{KT}, it leads to a coherent theory of Diophantine approximation over $\Q_p$. In this paper we settle this conjecture in full and prove several related results. Recently, there has been considerable progress in understanding $p$-adic Diophantine approximation on manifolds. We refer the reader to \cite{DG3, DG2, DG1} for its account.

One of the main technical points of \cite{KLW} is Proposition 7.3. This establishes that certain maps are ``good" with respect to friendly measures. We provide all relevant definitions later in the paper. The proof of \cite[Proposition 7.3]{KLW} crucially uses the Mean Value Theorem which is unavailable in the $p$-adic setting and thus represents the main barrier to establishing the aforementioned conjectures in the $p$-adic case. In this paper we develop a new technique that bypasses this barrier and thus enables us to develop a coherent theory of Diophantine approximation for friendly measures on $\Q_p^n$.


\section{Main results}\label{mainresults}

While postponing the definitions that we use until \cref{preli}, we now state our main results. In what follows, given a measure $\mu$ on an open subset $U$ of $\Q_\nu^d$ and a map $\f: U\to \Q_\nu^n$, $\f_\star\mu$ will denote the corresponding pushforward measure on $\Q_\nu^n$, where $\nu$ is either a prime number $p$ or $\infty$. Thus $\Q_\nu$ is either $\Q_p$ or $\R$. Our first main theorem reads as follows.

\begin{theorem}\label{ff}
Let $U$ be an open subset of $\Q_\nu^d$, and $\f: U\to \Q_\nu^n$. Let $\mu$ be an absolutely decaying Federer measure on $U$. Suppose that $\f$ is a $C^{l+1}$ map that is nonsingular and $l$-nondegenerate at $\mu$-almost every point. Then $\f_\star\mu$ is friendly.
\end{theorem}

One of the main results in \cite{KLW} is the above theorem for $\nu=\infty$. Thus \cref{ff} is new when $\nu$ is a prime number $p$. As a corollary of \cref{ff} we prove the following theorem which resolves the conjecture of Kleinbock and Tomanov appearing as Conjecture~FP in \cite{KT}.

\begin{theorem}\label{FP}
Let $\mu$ be a self-similar measure on the limit set of an irreducible family of contracting similitudes of $\Q_\nu^d$ satisfying the open set condition, and let $\f: \Q_\nu^d\to \Q_\nu^n$ be a smooth map which is nonsingular and nondegenerate at $\mu$-almost every point. Then $\f_\star\mu$ is strongly extremal.
\end{theorem}

We recall that a measure $\mu$ on $\Q_{\nu}^n$ is called strongly extremal if $\mu$-almost
every point is not ``very well multiplicatively approximable''. This implies that $\mu$ is extremal, i.e. that $\mu$-almost every point is not ``very well approximable''. We refer the reader to Section 11 of \cite{KT} for the definitions of these Diophantine properties.

Another main theorem of this paper is the $p$-adic analogue of the results of Kleinbock \cite{Kleinbock-exponent} regarding Diophantine approximations on affine subspaces. For the definitions of Diophantine exponents $\omega(\cdot)$ of measures and subspaces, the reader is referred to \cref{dioex} and \cite{DG2}.

 \begin{theorem} \label{affine}
 \label{DG} Suppose $\cL$ is a $d$-dimensional affine subspace of $\Q_\nu^n$.
  Let $\mu$ be an absolutely decaying and Federer measure on $\Q_\nu^d$ and let $\f: \Q_\nu^d \to \cL$ be a smooth map which is nondegenerate in $\cL$
at $\mu$-almost every point of $\Q_\nu^d$. Then $\omega(\f_\star\mu) = \omega(\cL)$.

\end{theorem}

When $\nu=\infty$, the above theorem is one of the main results in \cite{Kleinbock-exponent}. When $\nu$ is a prime number, \cref{affine} is new.
\subsection{Difficulties in  $p$-adic fields}
We want to point out the main difficulties in proving \cref{ff} in an ultrametric set-up.

\begin{itemize}
	\item There is no analogue of the Mean Value Theorem for $p$-adic continuously differentiable maps, not even for $C^1$ maps defined in \cref{function}. But the proof in case of $\R$ (\cite{KLW}) uses the Mean Value Theorem in a crucial manner.
	\item In order to bypass the use of the Mean Value Theorem, one may use the definition of $C^k$ maps. It is here that the difference quotients $\bar\Phi_\beta$ as in \cref{function}, kick into the picture. Note that $\bar\Phi_\beta$ is a multivariable function, even for maps from $\Q_\nu$ to $\Q_\nu$. This makes the analogue of many crucial lemmata in \cite{KLW} very different.
	\item The relation between $\bar\Phi^k f'$ with
	$f^{k+1}$ for any $C^{k+1}$ function from $\Q_\nu$ to $\Q_\nu$, poses diffficulties in various parts of the proof.
	\item For maps in higher dimensions, we use combinatorial arguments throughout the paper. For instance, we need to know the exact error terms in the higher dimensional ``Taylor theorem" in $\Q_\nu^d$. This seems to be unavaliable in the literature and we provide the necessary arguments in \cref{appen}. \item For the convenience of the reader, we have provided a detailed discussion of the main obstacles and our strategy in overcoming them in the simplest non-trivial case in \cref{Sketch}. 
\end{itemize}
We conclude this section with some remarks.
\subsection*{Remarks}
\begin{enumerate}
    \item  There has also been recent interest in the question of Diophantine approximation in positive characteristic, when one approximates Laurent series by ratios of polynomials. The positive characteristic version of the Baker-Sprind\v{z}huk conjectures were settled in \cite{G2}. The methods of this paper can be adapted to the function field setting in a relatively straightforward manner as they work over any ultrametric field.
    \item Similarly, it is possible to extent our main results to more than one place of $\Q$, namely to the $S$-arithmetic setting.
    \item In a forthcoming work, we will use the results in this paper to investigate badly approximable points on manifolds in $\Q_p$. In short, we establish the $p$-adic version of the main results of \cite{BNY22}.

\end{enumerate}

\section{Preliminaries}\label{preli}
In this section we recall some terminology and definitions from \cite{KLW, KT} and \cite{Kleinbock-exponent}. Our exposition necessarily follows these papers quite closely.

\subsection{Besicovitch spaces}\label{Besi} A metric space $X$ is called \emph{Besicovitch} \cite{KT} if there exists a constant $N_X$ such that the following holds: for any bounded subset $A$ of $X$ and for any family $\mathcal{B}$ of nonempty open balls in $X$ such that
   every  $x \in A$  is a center of some ball in $\mathcal B$,
 there is a finite or countable subfamily $\{B_i\}$ of $\mathcal B$ with
 \begin{equation}\label{besi_eqn} 1_A \leq \sum_{i}1_{B_i} \leq N_X. \end{equation}
{ Let us recall \cite[Lemma 2.3 ]{KT}. }
{
\begin{lemma}\label{crucial}
    For a metric space $X$, let us define
    $$M_X:=\sup\left\{k~|~\text{there are balls } B_i=B(x_i,r_i), 1\leq i\leq k, \text{ s.t } \cap_{i=1} B_i\neq\emptyset, x_i\notin \cup_{j\neq i} B_j \right\},$$
    and for any $c>1$, 
$$
D_X(c):= \sup\left\{k~\left|\begin{aligned}~&\text{there are } x\in X, r>0 \text{ and pairwise disjoint balls }\\ & B_1,\cdots,B_k \text{ of radius } r \text{ contained in } B(x,cr)\end{aligned}\right.\right\}.
$$
If both $M_X$ and $D_X(8)$ are finite, then $X$ is Besicovitch.
\end{lemma}}
{For any ultrametric space, any two balls either do not intersect, or one is contained inside another. Hence $M_X=1$ for any ultrametric space.}

{Now when $X=\Q_p$, since the balls are all of radius $p^n$ with $n\in \Z$, a ball of radius $p^{n}$ can contain at most $p$ many disjoint balls of radius $p^{n-1}$. Hence $D_X(8)<\infty$ for $X=\Q_p$ and applying \cref{crucial} it follows that $\Q_p$ is Besicovitch. One can also see that $\R^d$, $\Q_p^d$ are Besicovitch spaces; see \cite[Theorem 2.7]{Mattila} and \cite[Corollary 4.13]{Aldaz}. }

{ In particular, we will use later that $N_{\Q_p^d}=1$ since any two balls that intersect in this nonarchimedean space are contained in each other. Readers are also referred to \cite{DonneRigot2017,DonneRigot2019} for examples of several Besicovitch spaces.}

\subsection{Federer measures}

Let $\mu$ be a locally finite Borel measure on a metric space $X$, and $U$ be an open subset of $X$ with $\mu(U) > 0$. Following \cite{KLW} we say that $\mu$ is \textit{$D$-Federer on $U$} if
$$ \sup_{\substack{x \in \supp \mu,\, r > 0\\ B(x, 3r) \subset U}} \frac{\mu(B(x, 3r))}{\mu(B(x,r))} < D.$$
We say that $\mu$ as above is \textit{Federer} if, for $\mu$-almost every $x \in X$, there exists a neighbourhood $U$ of $x$ and $D > 0$ such that $\mu$ is $D$-Federer on $U$. We refer the reader to \cite{KLW} and \cite{KT} for examples of Federer measures.

\subsection{Nonplanar measures}

Suppose $\munu$ is a measure on $\Q_\nu^d$. We call $\munu$ \textit{nonplanar} if  $\munu(\mathcal L_1) = 0$ for any affine hyperplane $\mathcal L_1$ of $\Q_\nu^d$.
Let  $\f = (f_{1}, \cdots, f_{n})$ be a map from $\Q_\nu^d \to \Q_\nu^n$. The pair $(\f, \munu)$ will be called \textit{nonplanar at $x_0 \in \Q_\nu^d$} if, for any neighbourhood $B$ of $x_0$, the restrictions of $1, f_{1}, \cdots, f_{n}$ to $B \cap \supp \munu$ are linearly independent over $\Q_\nu$. This is equivalent to saying that $\f (B \cap \supp\munu)$ is not contained in any proper affine subspace of $\Q_\nu^d$. Note that for $d=n$, if $\mu$ on $\Q_\nu^n$ is  nonplanar then $(\mathbf{Id}, \mu)$ is nonplanar at $\mu$ almost every point.

Suppose $\f:X\subset \Q_\nu^d\to \mathcal{L}\subset \Q_\nu^n$ be a map, $\mu$ be a measure on $X$, and $\mathcal{L}$ be an affine subspace of $\Q_\nu^n$. Then $(\f,\mu)$ is called nonplanar in $\mathcal{L}$ if $$
\mathcal{L}=\langle \f(B\cap \supp{\mu})\rangle_a,
$$ for every nonempty ball $B$ such that $\mu(B)>0$; see \cite{Kleinbock-exponent}. Here $\langle M\rangle_a$ means the intersection all subspaces in $\Q_\nu^n$ containing $M.$

\subsection{$C^k$ functions in $\Q_\nu^d$}\label{function}
We recall the definition of $p$-adic $C^k$ functions, following closely the exposition in \cite{KT}. We refer the reader to \cite{Sc} for a detailed treatment of $p$-adic calculus. Let $f:U\subset \Q_{\nu} \to \Q_{\nu}$ be a function and assume that $U$ does not have any isolated points.
The {first difference quotient}
$\Phi^1f$ of
$f$ is the function of two variables given by
$$
\Phi^1f(x,y) := \frac{f(x) - f(y)}{x-y} \quad(x,y\in U,\ x\ne y)
$$
defined on $$\nabla^2 U:= \{(x,y)\in U\times  U\mid
x\ne y\}\,.
$$ We say that
$f$ is 
$C^1$ at $a$ if the limit $$\lim_{(x,y)\to(a,a)}\Phi^1f(x,y)$$
exists,
and that 
$f\in
C^1(U)$
if $f$ is $C^1$ at every point of $U$.

More generally, for $k\in \N$ set $$\nabla^k U :=
\{(x_1,\dots,x_k)\in U^{k}\mid  x_i\ne x_j \text{ for } i\ne j \}\,,$$
and define the $k$-th order difference quotient
$\Phi^kf:\nabla^{k+1} U\to \Q_\nu$ of $f$ inductively by $\Phi^0f := f$
and
$$\Phi^kf(x_1,x_2,\dots,x_{k+1}) :=
\frac{\Phi^{k-1}f(x_1,x_3,\dots,x_{k+1}) -
	\Phi^{k-1}f(x_2,x_3,\dots,x_{k+1})}{x_1-x_2}\,.
$$
We say that $f$ is $C^k$ if $\Phi^kf$ can be extended to a continuous function $\bar\Phi^{k}f:U^{k+1}\to \Q_\nu$.

Also, we say that $f$ is
$C^k$ at $a$
if the limit
$$
\bar\Phi^kf(a,\dots,a):=\lim_{(x_1,\dots,x_{k+1})\to(a,\dots,a)}\Phi^kf(x_1,\dots,x_{k+1})$$
exists. By \cite[Theorem 29.9]{Sc}  
$f\in C^k(U)$
if $f$ is
$C^k$ at every point of $U$.
%
It can be shown that $C^k$ functions $f$ are $k$
times differentiable, and that
$$
f^{(k)}(x)  = k!\bar\Phi^kf(x,\dots,x).
$$

\medskip

The definition of $C^k$ functions of  several
ultrametric variables follows along similar lines. If  $f$ is  a $\Q_{\nu}$-valued function on
$U_1 \times \cdots \times U_d$, where each $U_i$ is a subset
of
$\Q_{\nu}$ without isolated points, denote by $\Phi^k_if$ the $k$th
order difference quotient
of
$f$ with respect to the variable $x_i$, and, more generally, for a
multi-index $\beta = (i_1,\dots,i_d)$ let
$$\Phi_\beta f := \Phi_1^{i_1}\circ\dots\circ \Phi_d^{i_d} f.$$
It is not hard to
check that the composition can be taken in any order. For any fixed $j$ permutation of coordinates in $U_j^{i_j+1}$ does not change the function; see \cite[Lemma 29.2 (ii)]{Sc}. This
``difference quotient
of order $\beta$'' is defined on $\nabla^{i_1}U_1 \times \cdots
\times \nabla^{i_d}U_d$,
and as before we say that $f$ belongs to $C^k(U_1 \times \cdots \times U_d)$
if  for any multi-index $\beta$ with  $|\beta| :=
\sum_{j = 1}^d i_j\le k$,  $\Phi_\beta f$ is
extendable to a continuous function $\bar\Phi_\beta f:
U_1^{i_1+1} \times \cdots \times U_d^{i_d + 1}\to \Q_{\nu}$.  As in the
one-variable case,
one can show that partial derivatives $
\partial_\beta f :=\partial_1^{i_1}\circ\dots\circ \partial_d^{i_d}f$
of a  $C^k$ function  $f$ exist and are continuous as long as $|\beta|\le k$.
Moreover, one has
\begin{equation}\label{beta}
\partial_\beta f(x_1,\dots,x_d)  = \beta
!\bar\Phi_\beta f(x_1,\dots,x_1,\dots,x_d,\dots,x_d).
\end{equation}
where $\beta ! := \prod_{j = 1}^d i_j!$, and each of the variables
$x_j$ in the right hand side of
\cref{beta} is repeated $i_j + 1$ times.
Lastly, we recall the definition of a function being nonsingular at a point. 
\begin{definition}\label{defn_nonsing}
A function $C^1$ function $\f=(f_1,\cdots,f_n):U\subset \Q_\nu^d\to\Q_\nu^n$ is called nonsingular at a point $\bx_0$ if the matrix $\nabla\f(\bx_0)=[\partial_{e_i}f_j(\bx_0)]_{1\leq i\leq d, 1\leq j\leq n}$ has rank $d.$  
\end{definition}
\subsection{Nondegeneracy in an affine subspace}\label{nondeg}
Following \cite{Kleinbock-exponent}, we say that a map $\f : U \to \Q_\nu^n$, where $U$ is an open
subset of $\Q_\nu^d$, is $l$-nondegenerate in an affine subspace $\cL$ of $\Q_\nu^n$ at $\bx \in U$ if $\f(U) \subset \cL$, $\f$ is $C^l$ and the span of all the partial derivatives of $\f$ at $\bx$ up to order $l$ coincides with
the linear part of $\cL$. We say that $\f$ is nondegenerate in an affine subspace $\cL$ of $\Q_\nu^n$ at $\bx \in U$ if it is $l$-nondegenerate in an affine subspace $\cL$ of $\Q_\nu^n$ at $\bx$ for some $l\in\N$. We say that $\f$ is ($l$-)nondegenerate at $\bx\in U$ if the above hold with $\cL=\Q_\nu^n$.

If $\cM$ is a $d$-dimensional submanifold of $\cL$, we will say that
$\cM$ is nondegenerate in $\cL$ at $y\in \cM$ if any (equivalently, some) diffeomorphism $\f$
between an open subset $U$ of $\Q_\nu^d$ and a neighbourhood of $y$ in $\cM$ is nondegenerate
in $\cL$ at $\f^{-1}(y)$. We will say that $\f : U \to \cL$ (resp. $\cM \subset \cL$) is nondegenerate in $\cL$ if it is nondegenerate in $\cL$ at $\lambda$-almost every point of $U$, where $\lambda$ is the Lebesgue/Haar measure on $U$ (resp. of $\cM$, in the sense of the smooth measure class on $\cM$).

\subsection{Good maps and decaying measures}\label{sec3.5}
Let $\munu$ be a locally finite Borel measure on $\Q_{\nu}^d$ and $U$ be an open subset of $\Q_\nu^d$.
Given $A \subset \Q_\nu^d$ with $\munu(A)>0$ and a $\Q_\nu$-valued function $f$ on $\Q_\nu^d$, let
$$
\Vert f\Vert_{\mu,A}:=\sup_
{x\in A\cap\,\supp\munu}\vert f(\bx)\vert_{\nu}\,,
$$
where $|\cdot|_\nu$ is the standard $\nu$-adic absolute value.
A $\mu$-measurable function $f : \Q_\nu^d \to \Q_\nu$ is called $(C,\alpha)$-good on $U$ with respect to $\mu$ if for any open ball $B\subset U$ centered in $\supp(\mu)$ and $\varepsilon>0$ one has that
\begin{equation}\label{vb4}
\mu\left(\{\bx\in B~|~ \vert f(\bx)\vert_\nu <\varepsilon\}\right)\leq C\left(\frac{\varepsilon}{\Vert f\Vert_{\mu, B}}\right)^{\alpha}\mu(B).
\end{equation}
Similarly,
a function $f: \Q_{\nu}^d \to \Q_\nu$ is called  absolutely $(C,\alpha)$-good on $U$ with respect to $\mu$ if for any open ball $B\subset U$ centered in $\supp(\mu)$ and $\varepsilon>0$ one has
\begin{equation}\label{vb5}
\mu\left(\{\bx\in B~|~ \vert f(\bx)\vert_\nu <\varepsilon\}\right)\leq C\left(\frac{\varepsilon}{\Vert f\Vert_{B}}\right)^{\alpha}\mu(B)\,,
\end{equation}
where $\Vert f\Vert_{B}=\sup_{\bx\in B}|f(\bx)|_\nu$.
Since $\Vert f\Vert_{\mu, B}\leq \Vert f\Vert_{B}$, being absolutely $(C,\alpha)$-good implies being $(C,\alpha)$-good on $U$. The converse is true for measures having full support.

\begin{remark}\label{rem3.1}\rm
Note that if $f$ is (absolutely) $(C,\alpha)$-good w.r.t. $\mu$, \cref{vb4} and \cref{vb5} will remain true if the strict inequality in the left hand side is replaced by a non-strict inequality. Indeed, we trivially have that $\{\bx\in B~|~ \vert f(\bx)\vert_\nu \le \varepsilon\}\subset \{\bx\in B~|~ \vert f(\bx)\vert_\nu <\varepsilon+1/N\}$ for $N>1$. Then, for example in the case of absolutely  $(C,\alpha)$-good functions, we get that
$$
\mu(\{\bx\in B~|~ \vert f(\bx)\vert_\nu \le \varepsilon\})\le \mu(\{\bx\in B~|~ \vert f(\bx)\vert_\nu <\varepsilon+1/N\})\le
C\left(\frac{\varepsilon+1/N}{\Vert f\Vert_{B}}\right)^{\alpha}\mu(B)
$$
and it remains to let $N\to\infty$.
\end{remark}

\medskip

Next, let $\f=(f_1,\cdots,f_n): \Q_{\nu}^d \to \Q_{\nu}^n$ be a map, and $\mu$ be a locally finite Borel measure on $\Q_{\nu}^d$. We say that $(\f,\mu)$ is $(C,\alpha)$-good at $\bx$ if there exists a neighborhood $U$ of $\bx$ such that any linear combination of $1,f_1,\cdots,f_n$ is $(C,\alpha)$-good on $U$ with respect to $\mu$.
We will simply say $(\f,\mu)$ is good at $\bx$ if it is $(C,\alpha)$-good at $\bx$ for some positive $C$ and $\alpha$.
We will say $(\f,\mu)$ is good (on $U$) if it is good at $\mu$-almost every point $\bx$ (of $U$). When $\mu$ is Lebesgue/Haar measure on $\Q_\nu^d$, we omit `with respect to $\mu$' and just say that the map $\f$ is good (at $\bx$). As is well known, polynomials are $(C,\alpha)$-good at every point for some $C$ and $\alpha$ that depend only on $d$ and the degree of the polynomial in question, see \cite[Lemma~4.1]{GT}. In fact, there are many examples of good maps. We recall Lemma~2.5 from \cite{KM} and Proposition~4.2 from \cite{KT} which show that nondegenerate maps are good. For the definition of $C^l$ maps in the $p$-adic case, the reader is referred to \cref{function}.

\begin{proposition}[{See \cite[Lemma~2.5]{KM} and \cite[Proposition~5.1]{KT}}]\label{prop3.2}
Let $\f = (f_1,\cdots, f_n)$ be a $C^l$ map from an open subset $U \subset \Q_{\nu}^d$
to $\Q_{\nu}^n$
which is $l$-nondegenerate in $\Q_{\nu}^n$ at $\bx_0 \in U$. Then there is a neighbourhood $V \subset U$ of $\bx_0$
such that any linear combination of $1, f_1\cdots f_n$ is
$(C',\alpha)$-good on
$V$, where $C',\alpha>0$ only depends on $d,l$ and the field. In particular, the nondegeneracy of $\f$ in $\Q_{\nu}^n$ at $\bx_0$  implies that $\f$ is good at $\bx_0$.
\end{proposition}

\noindent We further have the following corollary whose proof is identical to that of \cite[Corollary~3.2]{Kleinbock-extremal}.

\begin{corollary}
Let $\cL$ be an affine subspace of $\Q_{\nu}^n$ and let $\f = (f_1,\cdots, f_n)$
be a map from an open subset $U$ of $\Q_{\nu}^d$ to $\cL$ which is nondegenerate
in $\cL$ at $\bx_0 \in U$. Then $\f$ is good at $\bx_0$.
\end{corollary}

Given $C, \alpha > 0$ and an open subset $U$ in  $\Q_\nu^n$, we say that $\munu$ is (absolutely) \textit{$(C, \alpha)$-decaying on $U$} if any affine map is (absolutely) \textit{$(C, \alpha)$-good on $U$} w.r.t. $\mu$.

\subsection{Friendly measures}
We are now ready to define \textit{friendly} measures following \cite{KLW}.
\begin{definition}[Friendly measure]
A locally finite Borel measure $\munu$ on $\Q_\nu^n$ is called friendly if $\munu$ is nonplanar and for $\munu$-almost every $\bx\in \Q_\nu^n$, there exist a neighbourhood $U$ of $\bx$ and positive $C,\alpha, D > 0$ such that $\munu$ is $D$-Federer on $U$ and $\munu$ is $(C, \alpha)$-decaying on $U$.
\end{definition}

Examples of friendly measures were constructed in \cite{KLW} and their $p$-adic analogues were desribed in \cite{KT}. We briefly recall the discussion from \cite[Section~11.6]{KT}. Let ``$\dist$'' be a metric on $\Q_\nu^n$ induced by $|\,\cdot\,|_\nu$. A map $h: \Q_{\nu}^n \to \Q_{\nu}^n$ is a contracting similitude with contraction rate $\rho$ if
$0 < \rho < 1$ and
$$ \dist(h(\bx),h(\by))= \rho \dist(\bx,\by) \quad\text{ for all } \bx, \by \in \Q_{\nu}^n. $$

By a result of Hutchinson \cite{Hutchinson}, for any finite family $h_1,\dots, h_m$ of contracting similitudes
there exists a unique nonempty compact set $Q$, called the limit set of the family,
such that
$$Q = \bigcup_{i=1}^{m}h_i(Q). $$

The family $h_1,\dots, h_m$ are said to satisfy the open set condition if there exists an open
subset $U \subset \Q_{\nu}^n$ such that
$$h_i(U) \subset U \quad\text{ for all } i = 1, \dots, m,$$
and
$$i \neq j \quad\Longrightarrow \quad h_i(U) \cap h_j(U) = \emptyset. $$

By a result of Hutchinson \cite{Hutchinson}, if $h_i, i = 1,\dots, m$, are contracting similitudes
with contraction rates $\rho_i$ satisfying the open set condition, and if $s > 0$ is the
unique solution to the equation $\sum_{i}\rho_{i}^{s} = 1$ then the $s$-dimensional Hausdorff measure $\mathcal{H}^s$ of $Q$ is positive and finite. The number $s$ is called the similarity dimension of the family $\{h_i\}$.

The family $\{h_i\}$ is irreducible if there does not exist a finite $\{h_i\}$-invariant collection
of proper affine subspaces of $\Q_{\nu}^n$. The following proposition is slightly stronger than \cite[Proposition 11.3]{KT}, which in turn is a direct adaptation of  \cite[Theorem 2.3]{KLW}. It provides a rich class of examples of friendly measures.

\begin{proposition}\label{simin}
Let $\{h_1, \dots, h_m\}$ be an irreducible
family of contracting similitudes on $\Q_{\nu}^n$
satisfying the open set condition, $s$ its similarity
dimension, and $\mu$ the restriction of $\mathcal{H}^s$ to the limit set $Q$ of the family $\{h_i\}$. Then $\mu$ is absolutely decaying and Federer.
\end{proposition}

\subsection{Diophantine exponents}\label{dioex}
For $\by\in\Q_\nu^n$, we define 
$$
\omega(\by)=\sup\{\omega~|~\vert \bq\cdot\by+q_0\vert_\nu\leq \frac{1}{\Vert (\bq, q_0)\Vert_\infty^\omega}\text{ for infinitely many } (\bq,q_0)\in\Z^{n+1}\setminus 0\}.$$ Here $\Vert \cdot\Vert_\infty$ is the max norm in $\R^{n+1}.$
For a Borel measure $\mu$ on $\Q_\nu^{n}$, from  \cite{Kleinbock-exponent} we recall the Diophantine exponent $\omega(\mu)$ of $\mu$ to be 
\begin{equation*}
\omega(\mu) = \sup\{ v~:~ \mu(\{y~|~\omega(y) > v\})>0\}.
\end{equation*}
If $\f:U\subset \Q_\nu^d\to \Q_\nu^n$ is a map such that, $\mathcal{M}=\f(U)$, where $U$ is an open subset of $\Q_\nu^d$, we define $\omega(\mathcal{M})=\omega(\f_\star\lambda)$. Here $\lambda$ is the Haar measure on $\Q_\nu^d.$

\section{Proofs of \cref{ff}, \cref{FP} and \cref{DG} }

We need the following proposition in order to prove \cref{ff} and \cref{DG}.

\begin{proposition}\label{mainp}
Let $\f=(f_1,\cdots,f_n): U\subset \Q_\nu^d\to \Q_\nu^n$ be a $C^{l+1}$ map, and let $\bx_0\in U$ be such that $\f$ is $l$-nondegenerate at $\bx_0$. Let $\mu$ be a locally finite Borel measure which is $D$-Federer and absolutely $(C,\alpha)$-decaying on $U$ for some $D,C,\alpha>0$. Then there exists a neighbourhood $V\subset U$ of $\bx_0$ and a positive $C'>0$ such that any
\begin{equation}\label{vb1}
g\in \mathcal{L}_{\f}:=\left\{c_0+\sum_{i=1}^n c_i f_i:c_0,\dots,c_n\in\Q_\nu\right\}
\end{equation}
is absolutely $(C',\frac{\alpha}{(2^{l+1}-2)})$-good on $V$ with respect to $\mu$.
\end{proposition}

This proposition is the key to establishing the theorems stated in \cref{mainresults} and requires new technical ideas. We will prove it in the following sections. When $\nu=\infty$, \cref{mainp} was proved in \cite[Section 7]{KLW}. One of the crucial facts used in \cite{KLW} was the Mean Value Theorem. Since there is no Mean Value Theorem in the $p$-adic setting, we have come up with a different proof.

The following lemma gives a sufficient condition written in terms of $\mu$, for $\f_\star\mu$ to be friendly.

\begin{lemma}\label{lemmmm}
Let $\mu$ be a $D$-Federer measure on an open subset $U$ of $\Q_\nu^d,$ let $\f:U \to \Q_\nu^n$ and $C,K,\alpha>0$ be such that
\begin{itemize}
    \item for any $\bx_1,\bx_2\in U$ one has that
    $$
    \frac{1}{K}\dist(\bx_1,\bx_2)\leq \dist(\f(\bx_1),\f(\bx_2))\leq K \dist(\bx_1,\bx_2);
    $$
    \item any $f\in \mathcal{L}_{\f}$ is absolutely $(C,\alpha)$-good on $U$ with respect to $\mu$, where $\mathcal{L}_{\f}$ as in \cref{vb1}.
\end{itemize}
Then $\f_\star\mu$ is friendly.
\end{lemma}
When $\nu=\infty$ the above lemma was proved in of \cite[Lemma~7.1]{KLW} and the same proof applies verbatim when $\nu$ is a prime number.

\subsection{ Proof of \cref{ff} using \cref{mainp} and \cref{lemmmm}}
Since $\f$ is a nonsingular $C^{l+1}$ map at $\bx_0$, we can find a neighbourhood $V$ of $\bx_0$ such that $\f|_{V}$ is bi-Lipschitz; see \cref{bi}. This satisfies the first condition of \cref{lemmmm}. By \cref{mainp}, we have that $\f$ and $\mu$ satisfy the second condition of \cref{lemmmm}. Therefore, by \cref{lemmmm} we have that $\f_\star\mu$ is friendly.

\subsection{ Proof of \cref{FP} using \cref{ff}}
\cref{FP} follows by combining  \cref{simin} and \cref{ff} and \cite[Corollary 11.2]{KT}
\subsection {Proof of \cref{DG} using \cref{mainp}}
We first state the following theorem which is \cite[Corollary $6.2$]{DG2}.
\begin{theorem}[{\cite[Corollary $6.2$]{DG2}}]
Let $\mu$ be a Federer measure on a Besicovitch metric space $X$, $\mathcal{L}$ an affine subspace of $\Q_\nu^n$, and let $\f : X\to \mathcal{L}$ be a continuous map such that $(\f, \mu)$ is good and nonplanar in $\mathcal{L}$. Then
\begin{equation}
\omega(\f_*\mu)= \omega(\mathcal{L})= \inf\{ \omega(\by)~|~\by\in\mathcal{L}\}= \inf\{\omega(\f(x)) ~|~ x\in\supp\mu\}.
\end{equation}
\end{theorem}

\begin{remark}
Let $\mu$ be a
friendly measure on $\Q_\nu^d$, and let $\f :\Q_\nu^d \to \cL$ be an affine isomorphism. Then by the above theorem, 
$\omega (\f_\star\mu)=\omega(\cL).$
\end{remark}

Since in \cref{DG}, $\f$ is nondegenerate in $\cL$ at $\mu$ almost all points, 
we get that $(\f,\mu)$ is nonplanar in $\cL$. If not, then there exists a ball $B$ such that $\mu(B)>0$ and an affine subspace $\mathcal{L}'$, $\mathcal{L}'\subset \mathcal{L}$ such that $\f(B\cap\supp\mu)\subset \mathcal{L}'.$ Let $\bx\in B\cap \supp{\mu}$ such that all partial derivatives of $\f$ at $\bx$ up to order $l$ spans the linear part of $\mathcal{L}.$ But the fact $\f(B\cap\supp\mu)\subset \mathcal{L}',$ implies all derivatives at $\bx$ must lie on the linear part of $\mathcal{L}'$ which is proper in $\mathcal{L}$ giving a contradiction.

Therefore, \cref{DG} follows from combining \cref{mainp} and the above theorem. The rest of the paper is spent on proving \cref{mainp}.

\section{Notes on the proof of \cref{mainp}}\label{Sketch}

In this section, we consider the simplest nontrivial situation in which we attempt to explain the obstacles we face and the ideas we use in the general case, where the obstacles and ideas are less transparent. Recall that to validate \cref{mainp} we need to demonstrate that under certain conditions on the measure $\mu$ and the map $\f=(f_1,\cdots,f_n):\Q_\nu^d\to\Q_\nu^n$, there exists a neighborhood $V$ such that all linear combinations of $1,f_1,\cdots,f_n$ are good w.r.t $\mu$, a concept stated in \cref{sec3.5}.

\subsection{Sufficient conditions}
The task is that given any function $g:B\subset\Q_\nu^d\to \Q_\nu$, defined on a ball $B$ of radius $r$, we want to find sufficient conditions on $g$, such that $g$ is good in this ball $B$ w.r.t $\mu$; see \cref{good}. It is too ambitious to hope that the same conditions as in the real case, \cite[Theorem 7.6]{KLW} would be enough. Next we discuss why this is so.
 \begin{itemize}
     \item Let us constrain ourselves to the $d=1$ case. The lemmas \cite[Lemma 7.4]{KLW} and \cite[Lemma 7.5]{KLW}, which are crucial in \cite[Theorem 7.6]{KLW} are not clear in the $\Q_\nu$ setting. For instance, in \cite[Lemma 7.4]{KLW} it was shown that for any $C^1$ function on ball $B\subset\R$ of radius $r$, we have that
     $$ r\inf\limits_{x\in B} \vert g'(x)\vert \ll \sup\limits_{x\in B}\vert g(x)\vert.$$ This is a direct implication of the Mean Value Theorem.
     \item However, the above is not true\footnote{We are thankful to an anonymous referee who suggested the example we provide here.} in $\Q_\nu$. For instance, we may take $B=\Z_\nu$, and $g_N(\sum_{i\geq 0} a_i\nu^i)=\sum_{i\geq N} a
_i \nu^i$ for $N\in \N.$ Then $g_N$ differes from the identity map by a locally constant map. Hence for every $x\in\Z_\nu,$ $\vert g'_{N}(x)\vert_\nu=1,$ whereas $\sup_{x\in B} \vert g_{N}(x)\vert_\nu\leq \nu^{-N}$, by definition. Hence for sufficiently large $N$, the above inequality is false. Nevertheless, we can use the definition of $C^k$ functions; see \cref{function} in order to prove a weaker statement; see \cref{lemn}. In particular, for any $C^1$  function defined on a ball $B\subset\Q_\nu$ of radius $r$,  we show that
     $$ r \inf\limits_{(x,y)\in B\times B} \vert \bar\Phi^1 g(x,y)\vert_\nu\ll \sup\limits_{x\in B}\vert g(x)\vert_\nu.$$
     \item For the same reason as above our \cref{lem2} is weaker than \cite[Lemma 7.5]{KLW}.
     \item Other than having two weaker lemmas, there is another issue that we face. In the real case, we have that $(g')'=g''$ for any $C^2$ function. But in case of any $C^2$ function $g$ defined on $\Q_\nu$, we have instead that
         $$
     \bar\Phi^1 g'(x,y)=\bar\Phi^2 g(x,x,y)+\bar\Phi^2 g(x,y,y).
     $$
     Hence it is not immediately clear what sort of relation we have between $$ \inf\limits_{(x,y,z)\in B\times B\times B}\vert \bar\Phi^2 g(x,y,z)\vert_\nu \quad\text{and}\quad \inf \limits_{(x,y)\in B\times B}\vert \bar\Phi^1 g'(x,y)\vert_\nu.$$
     Since the first one is defined over a larger domain, one would expect that the first quantity should be less than the second quantity upto some constant multiple. But the second quantity is the infimum of a sum of two functions, hence there is a possibility that it becomes $0$.
 \end{itemize}
 To tackle the aforementioned obstacles we impose three conditions on $g$ in \cref{good}, namely \cref{eq0}, \cref{eq1b} and \cref{upper}. Since these conditions are a bit technical, we refrain from stating them here, except for the most important one namely  \cref{eq0} which is discussed below. The latter two conditions showed up, albeit in a weaker sense, in the real setting. The first condition \cref{eq0} however, is entirely new and an important idea in our paper. No analogous version was needed in the proof of the real case. The second and third conditions are stronger than the respective conditions in the real case, \cite[Theorem 7.6]{KLW}. \\
 \indent As we prove the main theorem \cref{ff} with no extra assumption than those are present in the real case, we need to show that these stronger sufficient conditions in \cref{good} are satisfied by any linear combination of $1,f_1,\cdots,f_n$.

\subsection{The first condition} The first condition \cref{eq0} in \cref{good} is the hardest to check among the three conditions in the theorem. In the case $d=1$ and $C^k$ function $g$, note that the first condition is to check a chain of inequalities as follows:
 \begin{equation}\label{chainstar}\inf \vert\bar\Phi^k g \vert_\nu \ll\inf \vert \bar\Phi^{k-1}g'\vert_\nu\ll \cdots
    \ll \inf\vert \bar\Phi^{1}g^{k-1}\vert_\nu \ll \inf\vert g^{k}\vert_\nu.\end{equation}
Here and elsewhere in this section we adopt the convention that
$$\inf\vert\bar\Phi^j g\vert_\nu:=\inf\limits_{(x_1,\cdots,x_{j+1})\in B\times \cdots\times B}\vert\bar\Phi^j g(x_1,\cdots,x_{j+1})\vert_\nu\qquad\text{for $j=0,\dots, k$.}
$$

First, we show that polynomials satisfy all three conditions; see \cref{poly}.
We use the trick of `proof by maximality', which is also used in \cite[Lemma 7.7]{KLW}. But the proof in \cite[Lemma 7.7]{KLW} uses the Mean Value Theorem and there were only two weaker conditions to check unlike our case, where we need to verify three stronger conditions.

\subsection{A very simple case of polynomials}
Let us sketch how we handle the simplest nontrivial case of polynomials, where three conditions in \cref{good} are satisfied. Let $P$ be a polynomial over $\Q_\nu$ in one variable such that $\sup_{x\in B(0,1)}\vert P(x)\vert_\nu=1$. Suppose $k$ is the largest integer such that there exists $(x_1,\cdots,x_{k+1})\in B(0,1)^{k+1}$ with
$$
\vert \bar\Phi^{k} P(x_1,\cdots,x_{k+1})\vert_\nu>\frac{1}{\nu^{6\alpha+1}},
$$ where $\alpha>0$ is carefully chosen. Such a $k$ will exist since we started with a `normalized' polynomial. Now for simplicity, let $k=2$. Hence there exists $(a,b,c)\in B(0,1)^3$ such that
$$\vert \bar\Phi^2 P(a,b,c)\vert_\nu>\frac{1}{\nu^{6\alpha+1}}.$$ Then it is easy to see that for any $(x,y,z)\in B(0,1)^3$, $$ \bar\Phi^2 P(x,y,z)-\bar\Phi^2 P(a,b,c)= \bar\Phi^3 P(x,y,z,a)(z-a)+\bar\Phi^3 P(x,y,a,b)(y-b)+\bar\Phi^3 P(x,a,b)(x-c).$$ Now we use ultrametric norm property to conclude that for any $(x,y,z)\in B(0,1)^3$,
$$ \vert \bar\Phi^2 P(x,y,z)\vert_\nu\leq \max\{ \vert \bar\Phi^2 P(a,b,c)\vert_\nu, \vert \bar\Phi^3 P(x,y,z,a)\vert_\nu, \bar\Phi^3 P(x,y,a,b)\vert_\nu, \vert\bar\Phi^3 P(x,a,b)\vert_\nu\}.$$ By maximality of $k=2$, the max in the above should be attained by $\vert \bar\Phi^2 P(a,b,c)\vert_\nu$, giving

$$\vert \bar\Phi^2 P(x,y,z)\vert_\nu>\frac{1}{\nu^{6\alpha+2}}.$$
This verifies the second condition \cref{eq1b}.
Next we claim,
$$ \inf\vert \bar\Phi^2 P\vert_\nu\ll  \inf\vert \bar\Phi^1 P'\vert_\nu.$$
The following simple observation is key for verifying the first inequality in \cref{chainstar}. For any $(x,y,z)\in B(0,1)^3$, we have that
\begin{equation}
2\bar\Phi^2 P(x,y,z)=\bar\Phi^1 P'(x,y)+\bar\Phi^3 P(x,x,y,z)(z-x)+\bar\Phi^3 P(x,y,y,z)(z-y).
\end{equation}
Now let $\inf\vert\bar\Phi^1P'\vert_\nu=\vert\bar\Phi^1P'(c,d)\vert_\nu$ for some $c,d\in B(0,1).$
For any $z\in B(0,1),$
$$\vert 2 \bar\Phi^2 P(c,d,z)\vert_\nu\leq \max\{\vert \bar\Phi^1 P'(c,d)\vert_\nu, \vert\bar\Phi^3 P(c,c,d,z)\vert_\nu, \vert\bar\Phi^3 P(c,d,d,z)\vert_\nu\}.$$
If for some $z\in B(0,1), \vert \bar\Phi^1 P'(c,d)\vert_\nu\geq \max\{\vert\bar\Phi^3 P(c,c,d,z)\vert_\nu, \vert\bar\Phi^3 P(c,d,d,z)\vert_\nu\}$ then we are done. If not, then for every $z$,
$$ \vert 2 \bar\Phi^2 P(c,d,z)\vert_\nu\leq \vert\bar\Phi^3 P(c,c,d,z)\vert_\nu, \text{ or }  \vert 2 \bar\Phi^2 P(c,d,z)\vert_\nu\leq \vert\bar\Phi^3 P(c,d,d,z)\vert_\nu.$$ We choose $\alpha$ such that  either case will contradict the maximality.

\subsection{Approximations by polynomials}
Once we show that polynomials satisfy all three conditions in \cref{good}, we approximate difference quotients of functions of interest by difference quotients of polynomials in a uniform manner; see \cref{taylor}.

Consider the $d=1$ case. We show the following for any $C^{l+1}$ map $f$ defined on $B(y,r)\subset B(0,1)$, and for every $k\leq l$, \begin{equation}\label{star1}
\begin{aligned}
&\sup_{\substack{z_i\in B(y,r)\\
i=1,\cdots,k+1}} \vert \bar\Phi^k f(z_1,z_2,\cdots,z_{k+1})-\bar \Phi^k P_{f, y, l}(z_1,z_2,\cdots,z_{k+1})\vert_\nu \\
& \leq r^{l-k} \sup_{\substack{x_i,y_i\in B(y,r),\\
i=1,\cdots,l+1}} \vert \bar\Phi^l f(x_1,x_2,\cdots,x_{l+1})-\bar \Phi^l f(y_1,y_2,\cdots,y_{l+1})\vert_\nu.
\end{aligned}
\end{equation}
Here the $l$-th Taylor polynomial of $f$ at $y\in B$ is denoted as $$P_{f,y,l}(x)=f(y)+\bar \Phi^1f(y,y) (x-y)+\cdots+\bar\Phi^{l}f(y,\cdots,y)(x-y)^l.$$
We prove \cref{star1} by induction on $l$. The following is typically refered as ``Taylor's theorem'' in $\Q_\nu$,
\begin{equation}\label{taylorstar}
f(x)=P_{f,y,l}(x)+ \bar\Phi^{l+1}f(x,y,\cdots,y)(x-y)^{l+1},
\end{equation} which one can find in a simple manner, \cite[Theorem $29.4$]{Sc}.
In order to prove \cref{star1} by induction, we consider the function  $f_z(x):=\bar\Phi^1f(x,z)$ restricting one coordinate. We explore relations between $\bar\Phi^i P_{f,y,l}$ with $\bar\Phi^i P_{f_z,y,l}$; see \cref{chain rule}, which we then apply to prove \cref{taylor}.

In higher dimension, one would expect to have an analogue of \cref{taylorstar}, but we could not find any reference where the error term is explicitly calculated, which we require. So in an appendix \cref{appen} we prove Taylor's theorem in higher dimension with explicitly given error term.

\section{Good functions w.r.t. absolutely decaying measures}

\subsection{Notation}\label{notation}
First of all, in what follows if there is no risk of confusion, we will omit the subscript $\nu$ from the notation of $\nu$-adic absolute value, e.g. given $x\in\Q_\nu$, we will write $|x|$ for $|x|_\nu$. It will be clear from the context which absolute value we use. Throughout, the norm $\|\cdot\|$ on $\Q_\nu^i$ with $i\in\N$ induced by the $\nu$-adic absolute value will be the supremum norm. For any multi-index $\beta=(i_1,\cdots,i_d)$ we denote $|\beta|:=i_1+\cdots+i_d$ by abuse of notation.

Next, given a ball $B$ in $\Q_\nu$, $B^i:=B\times\cdots\times B\subset\Q_{\nu}^i$ will denote the product of $i$ copies of $B$ and $\by\in B^i$ will be written as $(y_j)$, thus $\by=(y_j)\in B^i$. We set $e_i=(0,\cdots,0,1,0,\cdots,0)\in \Z^d$, where $1$ is at the $i$-th place. Given a multi-index $\beta=(i_1,\ldots,i_d)$, we let $\beta(1):=\beta+\sum_{i=1}^d e_i$, thus $\beta(1)=(i_1+1,\ldots,i_d+1)$. We denote by $B(\by,r):=\{\bx~|~ \Vert \bx-\by\Vert\leq r\}$ a ball centred at $\by$ of radius $r$. Given a multi-index $\alpha=(i_1,\cdots,i_d)$, we also let $\mathbf{B}^{\alpha}=B(\by,r)^{\alpha}:=B(y_1,r)^{i_1}\times\cdots\times B(y_d,r)^{i_d}$ and denote by $\bx_{\alpha}$ the elements of $\mathbf{B}^{\alpha}$, thus $\bx_{\alpha}=(\bx_{1},\cdots,\bx_{d})\in \mathbf{B}^{\alpha}$, where $\bx_{j}\in B(y_j,r)^{i_j}$. For a function $g: X\to \Q_\nu$, we set
$$
\inf_{X}\vert g\vert:=\inf_{x\in X}\vert g(x)\vert.
$$
Similar notation will be used for supremum.
Finally, given $y\in\Q_\nu$ and $k\in \N$, we let $y^k:=(y,\cdots,y)\in \Q_\nu^k$ and, given a multi-index $\beta=(i_1,\cdots,i_d)$, we let $\bar{\by}_\beta:=(y_1^{i_1+1},\cdots,y_d^{i_d+1})$, thus $\bar{\by}_\beta$ is an element of $\Q_\nu^{i_1+1}\times \dots\times\Q_\nu^{i_d+1}$.
For a muti index $\beta=(i_1,\cdots,i_d)$, let us recall
$\beta ! := \prod_{j = 1}^d i_j!.$

\begin{lemma}\label{lemn}
    Let $B$ be a ball in $\Q_\nu$ of radius $r:=\nu^{-t}$, where and $t\in\N$. Suppose that  $f: B^d \to \Q_\nu$ is a $C^l$ map. Then for any multi-index $\beta=(i_1,\cdots,i_d)$ with $|\beta|:=i_1+\cdots+i_d\leq l$ we have that
    \begin{equation}\label{lemneq}
     \inf_{\nabla^{i_1}B\times\dots\times\nabla^{i_d}B}\vert\Phi_{\beta}f\vert\leq {\vert \beta!\vert ^{-1}}r^{-\vert\beta\vert}\Vert f\Vert_{B^d}\,,
\end{equation}
where $\vert\beta!\vert$ is the $\nu$-adic norm of $\beta!$, and $\Vert f\Vert_{B^d}=\sup_{\bx\in B^d}\vert f(\bx)\vert$.
\end{lemma}

\begin{proof}
Let $\bx=(x_1,\cdots,x_d)\in B^d$ and for $j=1,\cdots,d$ define $$D_{j}f(\bx):=f(x_1,\cdots,x_{j-1}, x_j+\nu^t,x_{j+1},\cdots, x_d)-f(\bx).$$ Suppose that $x\in B$, then $x+k\nu^t\in B$ for any integer $k\geq 1$. Then for any $\bx\in B^d,$
\begin{equation}\label{l1}
  \Phi_j^kf(x_1,\cdots, x_{j-1},x_j,x_j+\nu^t,\cdots,x_j+k\nu^t,x_{j+1},\cdots,x_d)=\frac{1}{k!\nu^{kt}} D_j^kf(\bx).
  \end{equation}
For $k=1$, the above equality follows from definition. Suppose, by the induction hypothesis, that the equality is true for $k-1\in\N$. This means that for any $\bx\in B^d$, \begin{equation*}\begin{aligned}
  &\Phi_j^{k-1}f(x_1,\cdots, x_{j-1},x_j,x_j+\nu^t,\cdots,x_j+(k-1)\nu^t,x_{j+1},\cdots,x_d)=\frac{1}{(k-1)!\nu^{(k-1)t}} D_j^{k-1}f(\bx),\\
  & \text{ and } \Phi_j^{k-1}f(x_1,\cdots, x_{j-1},x_j+\nu^t,\cdots,x_j+k\nu^t,x_{j+1},\cdots,x_d)\\
  & \vspace{3 in}=\frac{1}{(k-1)!\nu^{(k-1)t}}D_j^{k-1}f(x_1,\cdots,x_{j-1},x_j+\nu^t,x_{j+1},\cdots, x_d).
  \end{aligned}
  \end{equation*} Now using the two above equalities and definitions,  
  $$
  \begin{aligned}
     &k\nu^t\Phi_j^kf(x_1,\cdots, x_{j-1},x_j,x_j+\nu^t,\cdots,x_j+k\nu^t,x_{j+1},\cdots,x_d)\\&=\Phi_j^{k-1}f(x_1,\cdots, x_{j-1},x_j+\nu^t,\cdots,x_j+k\nu^t,x_{j+1},\cdots,x_d)-\\ & \vspace{3 in} \Phi_j^{k-1}f(x_1,\cdots, x_{j-1},x_j,x_j+\nu^t,\cdots,x_j+(k-1)\nu^t,x_{j+1},\cdots,x_d)\\
     &= \frac{1}{(k-1)!\nu^{(k-1)t}} \left(D_j^{k-1}f(x_1,\cdots,x_{j-1},x_j+\nu^t,x_{j+1},\cdots, x_d)- D_j^{k-1}f(\bx)\right)\\
     & = \frac{1}{(k-1)!\nu^{(k-1)t}} D_j^kf(\bx).
  \end{aligned}
  $$ Thus \cref{l1} is valid.

  Note that, by the ultrametric property, $\vert D_j^k f(\bx)\vert\leq \vert k!\vert^{-1} {\Vert f\Vert}_{B^d}$, where $\vert k!\vert$ is the $\nu$-adic norm of $k!\in \Z$. Hence taking norms on both sides of  \cref{l1} gives
  $$
  \vert k!\vert\nu^{-kt}\vert \Phi_j^kf(x_1,\cdots,x_{j-1},x_j,x_j+\nu^t,\cdots,x_j+k\nu^t,x_{j+1},\cdots,x_d)\vert\leq {\Vert f\Vert}_{B^d}.
  $$
Given a multi-index $\beta=(i_1,\cdots,i_d)$, let $D_\beta(\bx)=D_1^{i_1}\circ\cdots\circ D_d^{i_d}(\bx)$ and $\Phi_\beta=\Phi^{i_1}_1\circ\cdots\circ\Phi^{i_d}_d$. Therefore when $\vert \beta\vert =m$ we get that
    \begin{equation}\label{k1}
    D_\beta(\bx)=\beta!\nu^{mt}\Phi_\beta f(x_1,x_1+\nu^t,\cdots,x_1+i_1\nu^t,x_2,\cdots,x_2+i_2\nu^t,\cdots,x_d,x_d+\nu^t,\cdots,x_d+i_d\nu^t).
    \end{equation}
    Again, since $\vert D_\beta(\bx)\vert\leq \Vert f\Vert_{B^d}$, taking norms on both sides of \cref{k1} gives
    $$
     \inf_{\nabla^{i_1}B\times\cdots\times\nabla^{i_d}B}\vert\Phi_{\beta}f\vert\leq \vert \beta!\vert^{-1} r^{-\vert\beta\vert}\Vert f\Vert_{B^d}.$$

\end{proof}
 \begin{lemma}\label{lem2}
Let $C>0$ and $\alpha>0$ be given and let $\mu$ be an absolutely $(C,\alpha)$-decaying measure on $U\subset \Q_\nu^d$. Let $c,r,\varepsilon>0$, let $B(\by,r)\subset U$ be a ball centred at $\by=(y_j)\in \supp(\mu)$ and let $f:B(\by,r) \to \Q_\nu$ be a $C^l$ function with $l\geq 2$ such that
\begin{equation}\label{lower}
    \Vert \nabla f(\by)\Vert\geq c
\end{equation}
and
\begin{equation}\label{upper0}
    \Vert \bar\Phi_{\beta}f\Vert_{B(\by,r)^{\beta(1)}}\leq \frac{\varepsilon}{r^2}~~~~~ \qquad\text{for all multi-indices $\beta$ with } \vert\beta\vert=2\,,
\end{equation}
where $\beta(1)$ is defined in \cref{notation}. Then
\begin{equation}\label{local_estimate}
    \mu(\{ \bx\in B(\by,r)~ |~ \vert f(\bx)\vert<\varepsilon \})\leq C \left(\frac{\varepsilon}{cr}\right)^\alpha \mu(B(\by,r))\,.
\end{equation}
\end{lemma}
\begin{proof}
First, observe that for $\bx\in B(\by,r)$, by \cref{TAYLOR} and \eqref{defn:taylor}, we have that
$$f(\bx)=f(\by)+\sum_{i=1}^d(x_i-y_i)\bar\Phi_{e_i}f(\bar\by_{e_i})+\sum_{\substack{{\beta=(i_1,\cdots,i_d)}\\ \vert \beta\vert =2}}\bar\Phi_{\beta}f(\br^{\beta})\prod_{j=1}^d(x_{j}-y_{j})^{i_j} .\\
$$
Here the coordinates of $\br^\beta$ depend on $\bx$ and $\by$. Therefore
$$
\begin{aligned}
 & \vert  f(\bx)-f(\by)-\sum_{i=1}^d(x_i-y_i)\bar\Phi_{e_i}f(\bar\by_{e_i})\vert\\
 & \leq \max_{\vert \beta\vert =2}\vert \bar\Phi_{\beta}f(\br^{\beta})\vert \prod_{j=1}^{d}\vert x_j-y_j\vert^{i_j}\\
 & \stackrel{\eqref{upper0}}{\leq} \frac{\varepsilon}{r^2}r^2=\varepsilon.
\end{aligned}
$$
Therefore, by the ultra-metric property, for any $\bx=(x_1,\dots,x_d)$ and $\bx_0=(x^0_1,\dots,x^0_d)\in B(\by,r)$, we have that
\begin{equation}\label{vb2}
\vert f(\bx)-f(\bx_0)-\sum_{i=1}^d(x_i-x^0_i)\bar\Phi_{e_i}f(\bar\by_{e_i})\vert\le \varepsilon.
\end{equation}
Now take $\bx_0\in B(\by,r)$ for which $\vert f(\bx_0)\vert<\varepsilon$, otherwise there is nothing to prove. Then, by \cref{vb2}, for any $\bx\in B(\by,r)$ with $\vert f(\bx) \vert<\varepsilon$, we have $\left| \sum_{i=1}^d(x_i-x^0_i)\bar\Phi_{e_i}f(\bar\by_{e_i})\right|\le \varepsilon$.
Hence,
\begin{equation}\label{vb3}
\{\bx\in B(\by,r)~\left|~ \vert f(\bx)\vert<\varepsilon\right.\}\subset \left \{ \bx\in B(\by,r)~| ~ \left| \sum_{i=1}^d(x_i-x^0_i)\bar\Phi_{e_i}f(\bar\by_{e_i})\right|\le \varepsilon  \right\}.
\end{equation}
Note that $\bx\to \sum_{i=1}^d(x_i-x^0_i)\bar\Phi_{e_i}f(\bar\by_{e_i}) $ is an affine map. Therefore, since $\mu$ is absolutely $(C,\alpha)$-decaying on $U$, by definition, this affine map is absolutely $(C,\alpha)$-good on $B(\by,r)\subset U$ w.r.t. $\mu$. Further note that,
$$
\max_{\bx\in B(\by,r)}\left\vert\sum_{i=1}^d(x_i-x^0_i)\bar\Phi_{e_i}f(\bar\by_{e_i})\right\vert= r \max_{1\le i\le d}\left\vert \bar\Phi_{e_i}f(\bar\by_{e_i}) \right\vert\stackrel{\eqref{lower}}{\ge}rc.
$$
Hence, by the fact that the above affine map is absolutely $(C,\alpha)$-good on $B(\by,r)$ w.r.t. $\mu$, and in view of \cref{rem3.1}, we get that
 $$
\mu\left\{ \bx\in B(\by,r)~|~ \left| \sum_{i=1}^d(x_i-x^0_i)\bar\Phi_{e_i}f(\bar\by_{e_i})\right|\le \varepsilon \right\}\leq C\left(\frac{\varepsilon}{cr}\right)^\alpha \mu (B(\by,r))
$$
and \cref{vb3} now completes the proof.
\end{proof}

\begin{lemma}\label{lem3}
 Let $f:B^d\subset \Q_\nu^d \to \Q_\nu$ be a $C^{k+1}$ map, where $B$ is a ball  of radius $r>0$. Suppose that there exists $k\in\N$ and a multi-index $\beta=(i_1,\cdots,i_d)$ with $\vert\beta\vert=k\in\N$ such that for any two pairs of multi-indices $(\alpha_1,\alpha_2), (\alpha_1', \alpha_2')$,  with $\alpha_1+\alpha_2=\beta=\alpha_1'+\alpha_2'$ and $\vert \alpha_1\vert >\vert \alpha_1'\vert$, we have that
 \begin{equation}\label{eq0}
    \inf_{B^{\vert \alpha_1\vert+d}} \vert\bar\Phi_{\alpha_1} \partial_{\alpha_2}f \vert \leq C_{\alpha_1,\alpha_1'}\inf_{B^{\vert \alpha_1'\vert+d}} \vert \bar\Phi_{\alpha_1'}\partial_{\alpha_2'}f\vert
\end{equation}
and
\begin{equation}\label{eq1b}
   \inf_{B^{\vert \beta\vert +d}} \vert \bar\Phi_\beta f \vert > \frac{s}{r^k}\Vert f\Vert_{B^d}
\end{equation} for some $ C_{\alpha_1,\alpha_1'}\geq 1, s>0$.
Also, suppose that for some $S>0$, for every multi-index $\alpha$ with $\vert\alpha\vert\leq k+1$, 
\begin{equation}\label{upper}
\sup_{B^{\vert \alpha\vert+d}}\vert \bar\Phi_\alpha f \vert\leq \frac{S}{r^{\vert \alpha\vert}}\Vert f\Vert_{B^d}.
\end{equation} 
Then, if $i_j>0$, the function $\partial_{e_j}f$ also satisfies \cref{eq0}, \cref{eq1b} and \cref{upper} with multi-index $\beta-e_j$ with constants $\frac{s}{C_{\beta,\beta-e_j}S}$, $\frac{ C_{\beta,\beta-e_j}S}{s}$ in place of $s$ and $S$ respectively. Consequently we have that
\begin{equation}\label{e4}
    \frac{s \vert \beta!\vert}{C_{\beta,\beta-e_j}r}\Vert f\Vert_{B^d}< \Vert \partial_{e_j}f \Vert_{B^d}<\frac{S}{r}\Vert f\Vert_{B^d}.
\end{equation}
\end{lemma}
\begin{proof}By \cref{verybasic}, \begin{equation}\label{part}
\bar\Phi^i g'(x_1,\cdots, x_{i+1})= \sum_{n=1}^{i+1}\bar\Phi^{i+1} g(x_1,\cdots,x_{n-1},x_{n},x_{n}, x_{n+1},\cdots,x_{i+1}),\end{equation} for a $C^{j+1}$ map $g:\Q_\nu\to\Q_\nu$ and $i\leq j, j>0$.

Using \cref{part}  for $g=\partial_{e_j}f$ we have,
$$\begin{aligned}
& \bar\Phi_{\beta-e_j} \partial_{e_j}f(\bx_1,\cdots, \bx_{j-1},\bx_j,\bx_{j+1},\cdots, \bx_{d})\\
& = \sum_{n=1}^{i_j+1}\bar\Phi_\beta f(\bx_1,\cdots,\bx_{j-1},x_{j,1},\cdots,x_{j,n-1},x_{j,n},x_{j,n}, x_{j,n+1},\cdots,x_{j,i_j+1},\bx_{j+1},\cdots,\bx_d).\end{aligned}$$

Therefore \cref{lemn} gives the following:
\begin{equation}\label{inequalityimp}
    {\vert \beta!\vert^{-1}}\frac{1}{r^{k-1}}\Vert \partial_{e_j}f \Vert_{B^d} \stackrel{\eqref{lemneq}}{\geq} \inf_{B^{(k-1)+d}} \vert \bar\Phi_{\beta-e_j}\partial_{e_j}f\vert \stackrel{\eqref{eq0}}{\geq}\frac{1}{C_{\beta,\beta-e_j}} \inf_{B^{k+d}} \vert\bar\Phi_\beta f\vert\stackrel{\eqref{eq1b}}{\geq}\frac{s}{C_{\beta,\beta-e_j}r^{k}}\Vert f\Vert_{B^d}.
\end{equation}
{Above we have used the fact that $\vert \beta!\vert\leq\vert (\beta-e_j)!\vert$ if $i_j>1$ an integer, which follows from the ultrametric property of the norm on $\Q_\nu$}. By \cref{upper}, $\Vert \partial_{e_j}f\Vert_{B^d}\leq \frac{S}{r}\Vert f\Vert_{B^d}$. Combining this with \cref{inequalityimp} we have \cref{e4}.

\end{proof}

Note that $s\leq S$, that follows by taking $\alpha=\beta$ in \eqref{upper}, and comparing it to \eqref{eq1b}.
Define
$$
\mathcal{A}_{\alpha}:=\{(\alpha_1,\alpha'_1)~|~\exists~ \alpha_2,\alpha_2' \text{ multi-indices  such that }\alpha_1+\alpha_2=\alpha_1'+\alpha_2'=\alpha,\vert \alpha_1\vert >\vert \alpha'_1\vert\}.
$$

\begin{theorem}\label{good}
Suppose that $f:B^d\subset \Q^d_\nu \to \Q_\nu$ is a $C^{k+1}$ map, $B\subset\Q_\nu$ is an open ball of radius $r>0$ and that for some $k\in\N$ and multi-index $\beta$ with $\vert \beta\vert=k$, $f$ satisfies \cref{eq0}, \cref{eq1b} and \cref{upper}. Let $\mu$ be an absolutely $(C,\alpha)$-decaying measure on $B^d$. Then for any $\varepsilon<\min({s,\frac{s}{S}})$ one has
\begin{equation}\label{eqn vital}
    \mu\left(\left\{x\in B^d~|~\vert f(x)\vert <\varepsilon \Vert f \Vert_{B^d}\right \}\right)\leq C {2^k}\left( \prod_{(\alpha_1, \alpha_1')\in\mathcal{A}_{\beta}} C_{\alpha_1,\alpha_1'}\right)^{2^{k}\alpha} {\frac{1}{\vert \beta!\vert^\alpha}}\left(\frac{S}{s}\right)^{2^{k}\alpha} \varepsilon^{\eta_k\alpha} \mu(B^d),
\end{equation}
where $\eta_k=\frac{1}{2^{k+1}-2}$ for $k\geq 1$, and $\eta_0=\infty$ and ${0!=:1}$

\end{theorem}
\begin{proof}
Let us define the following set for $\delta>0$
$$
E(f,\delta,B^d):=\{ \bx\in B^d~|~ \vert f(\bx)\vert <\delta \}.$$
We will proceed by induction on $k$. If $k=0$, then by \eqref{eq1b} and $\varepsilon<s$, the left hand side set of \eqref{eqn vital} becomes empty, so the left hand side is $0$. On the other hand, $\mathcal{A}_{\beta}=\{0\}$ and $C_0\geq 1$, so the right hand side is infinite. Hence the conclusion of the theorem follows trivially in the case $k=0$. Suppose $\beta=(i_1,\cdots,i_d)$ and $i_j>0$. We know that $\partial_{e_j}f$ satisfies the conditions \cref{eq0}, \cref{eq1b} and \cref{upper}, by \cref{lem3}. Choose $$\delta =\varepsilon^{\frac{1}{2(\eta_{k-1}+1)}}.$$ Applying the induction hypothesis, we get \begin{equation}\label{one}
\begin{aligned}
     &\mu\left(E(\partial_{e_j}f,\delta\Vert \partial_{e_j}f\Vert_{B^d}, {B}^d)\right)\leq\\
     &  \hspace{1.5 in}C 2^{k-1}\left( \prod_{(\alpha_1, \alpha_1')\in\mathcal{A}_{\beta-e_j}} C_{\alpha_1,\alpha_1'}\right)^{2^{k-1}\alpha} {\frac{1}{\vert (\beta-e_j)!\vert^\alpha}}{\left(\frac{C_{\beta,\beta-e_j}^2S^2}{s^2}\right)}^{2^{k-1}\alpha}\delta^{\eta_{k-1}\alpha}\mu(B^d).
     \end{aligned}
 \end{equation}

 Let us choose $t=r\sqrt{\frac{\varepsilon}{S}}$. For every $\by\in {B}^d\cap \supp{\mu}\setminus E(\partial_{e_j}f, \delta \Vert \partial_{e_j}f\Vert_{B^d}, {B}^d),$ consider the ball $B(\by,t)$. By the choice of $\varepsilon$, and the fact that $s\leq S$ which we noted earlier, we have that $B(\by,t)\subset B^d$.
 Since $\by\notin E(\partial_{e_j}f, \delta \Vert \partial_{e_j}f\Vert_{B^d}, {B}^d)$,
 $$ \vert \partial_{e_j}f(\by)\vert \geq \delta \Vert \partial_{e_j}f\Vert_{B^d}$$ and, by \cref{upper} for every $\beta_\star$ multi-index, with $\vert \beta_\star\vert=2,$ we have $$
 \Vert \bar\Phi_{\beta_\star} f\Vert_{B(\by,t)^{{\beta_\star}(1)}}\leq \frac{S}{r^2}\Vert f\Vert_{B^d}=\frac{\varepsilon \Vert f\Vert_{B^d}S}{r^2\varepsilon}=\frac{\varepsilon \Vert f\Vert_{B^d}}{t^2}.$$ Now we can use \cref{lem2} and  \cref{e4} and conclude that
 $$\begin{aligned}
 &\mu(B(\by,t)\cap E(f,\varepsilon \Vert f\Vert_{B^d}, B^d))\\
 &\stackrel{\eqref{local_estimate}}{\leq} C\left(\frac{\varepsilon \Vert f\Vert_{B^d}}{\delta t \Vert \partial_{e_j}f\Vert_{B^d}}\right)^\alpha \mu(B(\by,t))\\
 &\stackrel{\eqref{e4}}{\leq} C {\left(\frac{\varepsilon C_{\beta,\beta-e_j} r}{\delta ts {\vert \beta!\vert}}\right)}^\alpha \mu(B(\by,t))\\
 & \leq C C_{\beta,\beta-e_j}^{\alpha} {\frac{1}{\vert \beta!\vert^\alpha}} {\left(\frac{ \sqrt{S}}{ s}\right)}^\alpha \frac{\varepsilon^{\frac{\alpha}{2}}}{\delta^\alpha} \mu(B(\by,t)).
\end{aligned}$$
By the Besicovitch covering property of $\Q_\nu^d$, one can cover ${B}^d\cap \supp{\mu}\setminus E(\partial_{e_j}f, \delta \Vert \partial_{e_j}f\Vert_{B^d}, {B}^d)$ by balls $B(\by,t)$ with multiplicity $1$. This means \eqref{besi_eqn} holds with $N_X=1, X=\Q_\nu^d$; see \S \ref{Besi}.  Hence
\begin{equation}
\label{e5}\begin{aligned}
&\mu\left(\left({B}^d\setminus E(\partial_{e_j}f, \delta \Vert \partial_{e_j}f\Vert_{B^d},{B}^d)\right)\bigcap E(f,\varepsilon \Vert f\Vert_{B^d}, B^d)\right)\\
& \leq  C C_{\beta,\beta-e_j}^\alpha {\frac{1}{\vert \beta!\vert^\alpha}}{\left(\frac{ \sqrt{S}}{ s}\right)}^\alpha \frac{\varepsilon^{\frac{\alpha}{2}}}{\delta^\alpha} \mu(B^d).\end{aligned}
\end{equation}
 Recall that $\delta =\varepsilon^{\frac{1}{2(\eta_{k-1}+1)}}$. Then combining \cref{one} and \cref{e5}, we have that
$$\begin{aligned}
&\mu({B}^d\cap E(f,\varepsilon \Vert f\Vert_{B^d}, B^d))& \\ & \leq
 C {2^{k-1}}\left( \prod_{(\alpha_1, \alpha_1')\in\mathcal{A}_{\beta-e_j}} C_{\alpha_1,\alpha_1'}\right)^{2^{k-1}\alpha} {\frac{1}{\vert (\beta-e_j
 )!\vert^\alpha}}{\left(\frac{C_{\beta,\beta-e_j}^2S^2}{s^2}\right)}^{2^{k-1}\alpha}\delta^{\eta_{k-1}\alpha}\mu(B^d)+ \\& \hspace{3.5 cm} C C_{\beta,\beta-e_j}^\alpha {\frac{1}{\vert \beta!\vert^\alpha}}{\left(\frac{ \sqrt{S}}{ s}\right)}^\alpha \frac{\varepsilon^{\frac{\alpha}{2}}}{\delta^\alpha} \mu(B^d)\\
& \leq C 2^{k-1}\left( \prod_{(\alpha_1, \alpha_1')\in\mathcal{A}_{\beta}} C_{\alpha_1,\alpha_1'}\right)^{2^{k}\alpha} {\frac{1}{\vert \beta!\vert^\alpha}} \left(\frac{S}{s}\right)^{2^{k}\alpha}\left(\delta^{\eta_{k-1}\alpha}+\frac{\varepsilon^{\frac{\alpha}{2}}}{\delta^\alpha}\right)\mu(B^d)\\
& \leq C {2^k}\left( \prod_{(\alpha_1, \alpha_1')\in\mathcal{A}_{\beta}} C_{\alpha_1,\alpha_1'}\right)^{2^{k}\alpha} {\frac{1}{\vert \beta!\vert^\alpha}}\left(\frac{S}{s}\right)^{2^{k}\alpha}\varepsilon^{\frac{\eta_{k-1}\alpha}{2(\eta_{k-1}+1)}}\mu(B^d).
\end{aligned}
$$
Hence, the theorem is proved with $\eta_k=\frac{\eta_{k-1}}{2(\eta_{k-1}+1)}$, which gives $ \eta_k=\frac{1}{2^{k+1}-2}$.

\end{proof}

\section{Polynomials satisfying the hypotheses in \cref{good}}

Let us denote by $\mathcal{P}_{d,l}$ the family of polynomials of degree $l$ in $d$ variables. 
\begin{lemma}\label{poly_maximal}
Suppose $P\in\mathcal{P}_{d,l}$ be a nonzero element and $\eta:\N\cup\{0\}\to\N$ be a function such that $\eta(n+1)\geq \eta(n)+1$. Suppose $k$ is the largest integer such that there exists $\beta$, a multi-index with $\vert \beta\vert =k$, with  $$
\vert \bar \Phi_{\beta} P(\bx_1,\cdots,\bx_{d})\vert >\frac{1}{\nu^{\eta(k)}}
$$ for some $(\bx_1,\cdots,\bx_d)\in  B(\mathbf{0},1)^{\beta(1)}$. Then for every $(\by_1,\cdots,\by_d)\in B(\mathbf{0},1)^{\beta(1)}$, \begin{equation}
    \vert \bar \Phi_{\beta} P(\by_1,\cdots,\by_{d})\vert >\frac{1}{\nu^{\eta(k)+1}}.
\end{equation}
\end{lemma}
\begin{proof}
Note 
$$\begin{aligned}
 &\bar\Phi_\beta P(\bx_1,\cdots,\bx_{d}) -\bar\Phi_{\beta} P(\by_1,\cdots,\by_d)\\
 & = \sum_{j=0}^{d-1}\bar\Phi_\beta P(\by_1,\cdots,\by_j,\bx_{j+1},\bx_{j+2},\cdots,\bx_{d}) -\bar\Phi_{\beta} P(\by_1,\cdots,\by_j,\by_{j+1},\bx_{j+2},\cdots,\bx_d)\\
 &=\sum_{j=0}^{d-1}\sum_{q=0}^{i_{j+1}} \bigg(\bar\Phi_{\beta} P(\by_1,\cdots,\by_j,x_{j+1,1},\cdots,x_{j+1,i_{j+1}+1-q},y_{j+1,1},\cdots, y_{j+1,q},\bx_{j+2},\cdots,\bx_d) \\
 & \hspace{2.5 cm}-\bar\Phi_{\beta} P(\by_1, \cdots,\by_j, x_{j+1,1},\cdots,x_{j+1,i_{j+1}-q},y_{j+1,1},\cdots,y_{j+1,q+1},\bx_{j+2},\cdots,\bx_d)\bigg)\\
 & =\sum_{j=0}^{d-1}\sum_{q=0}^{i_{j+1}} \bar\Phi_{\beta+e_{j}} P(\by_1,\cdots,\by_j,x_{j+1,1},\cdots,x_{j+1,i_{j+1}+1-q},y_{j+1,1},\cdots, y_{j+1,q+1},\bx_{j+2},\cdots,\bx_d) \delta_{j,q}
 \end{aligned}  
 $$ where {$\delta_{j,q}=(x_{j+1,i_{j+1}+1-q}-y_{j+1,q+1}).$}
 Therefore, we have $$\begin{aligned} & \vert \bar\Phi_{\beta} P(\bx_1,\cdots,\bx_{d})\vert -\vert\bar\Phi_{\beta} P(\by_1,\cdots,\bx_d)\vert \\
&\leq \max_{j=0}^{d-1}\max_{q=0}^{i_{j+1}}\left\{ \vert \bar\Phi_{\beta+e_j} P(\by_1,\cdots,\by_j,x_{j+1,1},\cdots,x_{j+1,i_{j+1}+1-q},y_{j+1,1},\cdots,y_{j+1,q+1},\bx_{j+2},\cdots,\bx_d)\vert\right\}.\end{aligned}$$
We proceed by contradiction. Suppose there exists a $(\by_1,\cdots,\by_d)\in B(\mathbf{0},1)^{\beta(1)}$ such that
 \begin{equation}
 \vert \bar\Phi_\beta P(\by_1,\cdots,\by_d)\vert\leq \frac{1}{\nu^{\eta(k)+1}}.
 \end{equation}
 Then we have $$\frac{1}{\nu^{\eta(k+1)}}<\frac{1}{\nu^{\eta(k)}}-\frac{1}{\nu^{\eta(k)+1}}<\vert \bar\Phi_{\beta} P(\bx_1,\cdots,\bx_{d})\vert -\vert\bar\Phi_{\beta} P(\by_1,\cdots,\bx_d)\vert.$$ The above will imply that there exists $0\leq j\leq d-1$, and $0\leq q\leq i_{j+1}$ such that $$\vert \bar\Phi_{\beta+e_{j}} P(\by_1,\cdots,\by_j,x_{j+1,1},\cdots,x_{j+1,i_{j+1}+1-q},y_{j+1,1},\cdots,y_{j+1,q+1},\bx_{j+2},\cdots,\bx_d)\vert >\frac{1}{\nu^{\eta(k+1)}}.$$ Since $\vert \beta+e_j\vert=k+1$, the last inequality contradicts the maximality of $k$.
\end{proof}

In the following theorem $\alpha, \alpha_1,\alpha'_1, \alpha_2,\alpha'_2$ and $\beta$ are all muti-indices.

\begin{theorem}\label{poly}
Fix $d$ and $l$. There exist $s,S>0$ and a collection $\bigcup_{\tilde\eta, \vert \tilde\eta\vert\leq l}\{C_{\alpha_1,\alpha'_1}: (\alpha_1,\alpha_1')\in \mathcal{A}_{\tilde{\eta}}\}$ of constants greater than $1$ such that for any nonzero $P\in \mathcal{P}_{d,l}$ the following are satisfied.
\begin{itemize}
\item There exists $k\leq l$ and a multi-index $\beta=(i_1,\cdots,i_d)$ with $\vert\beta\vert=k\in\N$ such that, for every  $\alpha_1+\alpha_2=\beta=\alpha_1'+\alpha_2'$,  $\vert \alpha_1\vert >\vert \alpha_1'\vert$,
 \begin{equation}\label{poly0}
    \inf_{B(0,1)^{\vert \alpha_1\vert+d}} \vert\bar\Phi_{\alpha_1} \partial_{\alpha_2}P\vert \leq C_{\alpha_1,\alpha_1'}\inf_{B(0,1)^{\vert \alpha_1'\vert+d}} \vert \bar\Phi_{\alpha_1'}\partial_{\alpha_2'}P\vert.
\end{equation}
and  \begin{equation}\label{poly1}
\inf_{B(0,1)^{\vert \beta\vert+d}}\vert \bar \Phi_{\beta} P(\bx_1,\cdots,\bx_{d})\vert > s \Vert P\Vert_{B(0,1)^d}.\end{equation}
\item
\begin{equation}\label{poly2} \sup _{B(0,1)^{\vert\alpha\vert+d}}\vert \bar \Phi_{\alpha} P(\bx_1,\cdots,\bx_{d})\vert<S \Vert P\Vert_{B(0,1)^d} ~ ~\forall~~ \vert \alpha\vert \geq 0.
\end{equation}
\end{itemize}
\end{theorem}

\begin{proof}
Without loss of generality we assume that $\Vert P \Vert_{B(0,1)^d}=1$. Let $a> 0$ be an integer such that $
\frac{1}{\nu^a}< \vert i\vert =\vert i\vert_\nu$ for $i=1,\cdots,l$. Consider the function $\eta:\N\cup\{0\}\to\N$, $\eta(n):=n^2a+na+1$. Let $k$ be the largest integer such that there exists a multi-index $\beta=(i_1,\cdots,i_d)$ with $\vert \beta\vert =k$ and  $(\bx_1,\cdots,\bx_{d})\in B(\mathbf{0},1)^{\beta(1)}$ such that
$$
\vert \bar \Phi_{\beta} P(\bx_1,\cdots,\bx_{d})\vert >\frac{1}{\nu^{\eta(k)}}. $$ { Such a $k$ exists because for $k=0$ the above inequality holds by the fact that $\Vert P\Vert_{B(0,1)^d}=1$ and $\eta(0)=1.$} { Also note that $k\leq l$}. By \cref{poly_maximal}, we have that \begin{equation}\label{lowerbound}\inf_{B(0,1)^{\vert \beta\vert+d}} \vert \bar \Phi_{\beta}P\vert>\frac{1}{\nu^{\eta(k)+1}}.\end{equation} If $k=0$ then \cref{poly0} will be true trivially. So, let us assume $k\neq 0$.\\

\textbf{Claim 1:} We claim that for any $j=1,\cdots,d$ if $i_j>0$ then \begin{equation}\label{claim}
\inf_{B(0,1)^{\vert \beta\vert+d}}\vert \bar\Phi_{\beta}P\vert\leq \vert i_j\vert^{-1}\inf_{B(0,1)^{\vert\beta\vert-1+d}}\vert \bar\Phi_{\beta-e_j}\partial_{e_j}P\vert .\end{equation}
Let $(\bx_1^0,\cdots,\bx_d^0)\in B(\mathbf{0},1)^{(\beta-e_j)(1)}$ be such that  $$ \inf_{B(0,1)^{\vert\beta\vert-1+d}}\vert \bar\Phi_{\beta-e_j}\partial_{e_j}P\vert= \vert \bar\Phi_{\beta-e_j}\partial_{e_j}P(\bx_1^0,\cdots,\bx_d^0)\vert.$$
Now by \cref{apn: lemma 1}, for any $(\bx_1,\cdots,\bx_d)\in B(\mathbf{0},1)^{\beta(1)}$,
\begin{equation}\label{claim0}\begin{aligned}
& i_j\bar \Phi_\beta P(\bx_1,\cdots,\bx_{d})\\
& = \bar\Phi_{\beta-e_j}\partial_{e_j}P(\bx_1,\cdots,\bx_{j-1},x_{j,1},\cdots,x_{j,i_j},\bx_{j+1},\cdots,\bx_d)\\
& ~~~~~ +\sum_{i=1}^{i_j} \bar\Phi_{\beta+e_j}P(\bx_1,\cdots,\bx_{j-1},x_{j,1},\cdots,x_{j,i-1},x_{j,i},x_{j,i},x_{j,i+1},\cdots,x_{j,i_j+1},\bx_{j+1},\cdots,\bx_d)(x_{j,i_j+1}-x_{j,i}).
\end{aligned}
\end{equation}
If for some $x_{j,i_j+1}\in B(0,1)$, $$\begin{aligned}
&\max_{i=1}^{i_j} \vert \bar\Phi_{\beta+e_j}P(\bx^0_1,\cdots,\bx^0_{j-1},x^0_{j,1},\cdots,x^0_{j,i-1},x^0_{j,i},x^0_{j,i},x^0_{j,i+1},\cdots,x^0_{j,i_j},x_{j,i_j+1},\bx^0_{j+1},\cdots,\bx^0_d)\vert\\
& \leq \vert \bar\Phi_{\beta-e_j}\partial_{e_j}P(\bx^0_1,\cdots,\bx^0_{j-1},x^0_{j,1},\cdots,x^0_{j,i_j},\bx^0_{j+1},\cdots,\bx^0_d)\vert\end{aligned},$$ then \cref{claim} holds using \cref{claim0}. If not, then for every $x_{j,i_j+1}\in B(0,1)$ there exists some $1\leq i\leq i_j$ such that $$\begin{aligned}
&\vert \bar\Phi_{\beta+e_j}P(\bx^0_1,\cdots,\bx^0_{j-1},x^0_{j,1},\cdots,x^0_{j,i-1},x^0_{j,i},x^0_{j,i},x^0_{j,i+1},\cdots,x^0_{j,i_j},x_{j,i_j+1},\bx^0_{j+1},\cdots,\bx^0_d)\vert\\
& > \vert \bar\Phi_{\beta-e_j}\partial_{e_j}P(\bx^0_1,\cdots,\bx^0_{j-1},x^0_{j,1},\cdots,x^0_{j,i_j},\bx^0_{j+1},\cdots,\bx^0_d)\vert.\end{aligned}$$ Since  $\vert i_j\vert>\frac{1}{\nu^a}$, we have that
$$\begin{aligned}
& \frac{1}{\nu^{\eta(k)+1+a}}<\frac{\vert i_j\vert }{\nu^{\eta(k)+1}} <\vert i_j\vert \vert \bar\Phi_\beta P(\bx_1,\cdots,\bx_d)\vert\\
& \ \ \ \ \ \ \ <\vert \bar\Phi_{\beta+e_j}P(\bx^0_1,\cdots,\bx^0_{j-1},x^0_{j,1},\cdots,x^0_{j,i-1},x^0_{j,i},x^0_{j,i},x^0_{j,i+1},\cdots,x^0_{j,i_j},x_{j,i_j+1},\bx^0_{j+1},\cdots,\bx^0_d)\vert.
\end{aligned}$$
Now note that $$\frac{1}{\nu^{(\eta(k)+1+a)}}> \frac{1}{\nu^{((k+1)^2a+(k+1)a+1)}}=\frac{1}{\nu^{\eta(k+1)}},$$ which will give a contradiction to the maximality of $k$. This yields Claim 1.\\

\textbf{ Claim 2:}
If $i_j>1$ then we claim that
    \begin{equation}\label{claim2}
\inf_{B(0,1)^{\vert \beta\vert-1+d}}\vert \bar\Phi_{\beta-e_j}\partial_{e_j}P\vert\leq \vert i_j-1\vert^{-1}\inf_{B(0,1)^{\vert\beta\vert-2+d}}\vert \bar\Phi_{\beta-2e_j}\partial_{2e_j}P\vert .
\end{equation}
Note that Claim 1 and \cref{lowerbound} imply that  \begin{equation}\label{lower1}\inf_{B(0,1)^{\vert \beta\vert-1+d}}\vert \bar\Phi_{\beta-e_j}\partial_{e_j}P\vert>\frac{1}{\nu^{\eta(k)+1+a}}.\end{equation}
Now observe that by \cref{apn: lemma 1}, for any $(\bx_1,\cdots,\bx_d)\in B(\mathbf{0},1)^{(\beta-e_j)(1)},$
$$\begin{aligned}
 &(i_j-1)\bar\Phi_{\beta-e_j}\partial_{e_j}P(\bx_1,\cdots,\bx_d)\\
 & = \bar\Phi_{\beta-2e_j}\partial_{2e_j}P(\bx_1,\cdots,\bx_{j-1},x_{j,1},\cdots,x_{j,i_j-1},\bx_{j+1},\cdots,\bx_{d})\\
 & +\sum_{i=1}^{i_j-1} \bar\Phi_{\beta}\partial_{e_j}P(\bx_1,\cdots,\bx_{j-1},x_{j,1},\cdots,x_{j,i-1},x_{j,i},x_{j,i},x_{j,i+1},\cdots,x_{j,i_j},\bx_{j+1},\cdots,\bx_d)(x_{j,i_j}-x_{j,i})\\
 & \stackrel{\cref{verybasic}}{=} \bar\Phi_{\beta-2e_j}\partial_{2e_j}P(\bx_1,\cdots,\bx_{j-1},x_{j,1},\cdots,x_{j,i_j-1},\bx_{j+1},\cdots,\bx_{d})\\
&\ \ \ \ +\sum_{i=1}^{i_j-1} \left (\sum_{b=1}^{i_j}  \bar\Phi_{\beta+e_j}P(\bt_{b,i})\right)(x_{j,i_j}-x_{j,i}),
\end{aligned}
$$ where $\bt_{b,i}=(\bx_1,\cdots,\bx_{j-1}, x_{j,i}, x_{j,1},\cdots, x_{j,b},x_{j,b}, x_{j,b+1}, \cdots, x_{j,i_j}, \bx_{j+1},\cdots, \bx_{d})$ has coordinates in terms of $\bx$. {In the last line of the above equality, in particular, we use \cref{verybasic} to get
$$\bar\Phi_{\beta}\partial_{e_j}P(\bx_1,\cdots,\bx_{j-1},x_{j,1},\cdots,x_{j,i-1},x_{j,i},x_{j,i},x_{j,i+1},\cdots,x_{j,i_j},\bx_{j+1},\cdots,\bx_d)= \left (\sum_{b=1}^{i_j}  \bar\Phi_{\beta+e_j}P(\bt_{b,i})\right).$$}

Note that $\frac{1}{\nu^{\eta(k)+1+2a}}>\frac{1}{\nu^{\eta(k+1)}}$. Hence the
same argument as in Claim 1 using \cref{lower1} yields Claim 2.\\

\textbf{ Claim c :}
Continuing this iterative method we get that for any $j=1,\cdots,d$, if $i_j>c+1$ then, \begin{equation}\label{claimm}
\inf_{B(0,1)^{\vert \beta\vert-c+d}}\vert \bar\Phi_{\beta-ce_j}\partial_{ce_j}P\vert\leq \vert i_j-c\vert^{-1}\inf_{B(0,1)^{\vert\beta\vert-c-1+d}}\vert \bar\Phi_{\beta-(c+1)e_j}\partial_{(c+1)e_j}P\vert .\end{equation}

\textbf{ Claim $(\beta-e_j,\beta-e_j-e_m)$:}
We claim that \begin{equation}\label{claimejem} \inf_{B(0,1)^{\vert \beta\vert-1+d}}\vert \bar\Phi_{\beta-e_j}\partial_{e_j}P\vert \leq \vert i_m\vert^{-1} \inf_{B(0,1)^{\vert \beta\vert-2+d}}\vert \bar\Phi_{\beta-e_j-e_m}\partial_{e_j+e_m}P\vert. \end{equation}
If $i_m>0 $ and $i_j>0$ then by \cref{apn: lemma 1} and \cref{verybasic} for any $(\bx_1,\cdots,\bx_d)\in B(\mathbf{0},1)^{(\beta-e_j)(1)},$ one can write the following:
$$\begin{aligned}
 & i_m\bar\Phi_{\beta-e_j}\partial_{e_j}P(\bx_1,\cdots,\bx_d)\\
 & =\bar\Phi_{\beta-e_j-e_m}\partial_{e_j+e_m}P(\bx_1,\cdots,\bx_{m-1},x_{m,1},\cdots,x_{m,i_m},\bx_{m+1},\cdots,\bx_{j-1},x_{j,1},\cdots,x_{j,i_j},\bx_{j+1},\cdots,\bx_{d})\\
\ \ \ \ & +\sum_{i=1}^{i_m} \left (\sum_{b=1}^{i_j}  \bar\Phi_{\beta+e_m}P(\bt'_{b,i})\right)(x_{m,i_m+1}-x_{m,i}),
\end{aligned}
$$ where $$\bt'_{b,i}=(\bx_1,\cdots,\bx_{m-1}, \delta_{m,i},\bx_{m+1},\cdots,\bx_{j-1},\mathbb{\kappa}_{j,b},\bx_{j+1},\cdots,\bx_d), \text{ and }$$
$$\delta_{m,i}=(x_{m,1}, \cdots, x_{m,i},x_{m,i},x_{m,i+1},\cdots, x_{m,i_{m}+1}) \text{ and } \kappa_{j,b}= (x_{j,1}, \cdots, x_{j,b},x_{j,b},x_{j,b+1},\cdots,x_{j,i_j})$$ have coordinates in terms of $\bx$. The deduction of the above from \cref{apn: lemma 1} and \cref{verybasic} follows the same way as in (the first unnumbered equations) Claim 2.  Now the
same argument as in Claim 1 using \cref{lower1} yields Claim $(\beta-e_j,\beta-e_j-e_m)$.
One can continue with the same arguments and get \cref{poly0}.\\
\indent The upper bound, \cref{poly2} follows from the observation that $P\to \Vert P\Vert_{B(0,1)^d}$ and $P\to \max_\alpha\Vert \bar \Phi_\alpha P\Vert_{B(0,1)^{\vert \alpha\vert+d}}$ are two norms on $\mathcal{P}_{d,l}$, a finite dimensional vector space over $\Q_\nu$. Hence they must be equivalent.

{ It is clear from \eqref{lowerbound}, that we may take $s=\frac{1}{\nu^{\eta(l)+1}}$. Also note that $C_{\alpha_1,\alpha_1'},$ we find in the proof  depends on the indices of $\alpha_1,\alpha_1',\beta$, and since $\vert \beta\vert \leq  l $, these constants can be made only dependent on $l$ and $d$; see \eqref{claim}, \eqref{claim2}, \eqref{claimm}, \eqref{claimejem}.}
\end{proof}

\section{Approximation by polynomials}
Let us denote $y^k:=(y,\cdots,y)\in \Q_\nu^k$ for $y\in\Q_\nu$, and $k\in \N$. We denote $\bar{\by}_\beta:=(y_1^{i_1+1},\cdots,y_d^{i_d+1})$, where $\beta=(i_1,\cdots,i_d)$. Suppose that $f:B^d\subset B(0,1)^d\to \Q_\nu$ is a $C^{l+1}$ map for some $l\in \N$. 
For a multi-index $\beta:=(i_1,\cdots,i_d)$ we set $L_{\beta,\by}(\bx):=\prod_{j=1}^d(x_j-y_j)^{i_j}$, where $\by\in\Q_\nu^d$. When $\beta=0$, we define $L_{\beta,\by}=1$. Note that 
$$
\bar\Phi_{\beta'} L_{\beta,\by}= 1, \text{ when } \beta=\beta', 
$$ and 
$$
\bar\Phi_{\beta'} L_{\beta,\by}=0, \text{ when } \beta\neq \beta', \vert \beta\vert=\vert\beta'\vert.
$$

We denote the $l$-th Taylor polynomial of $f$ at $\by\in B^d$ as  \begin{equation}\label{defn:taylor}\begin{aligned}
& P_{f,\by,l}(\bx):=f(\by)+ \sum_{k=1}^l\sum_{\vert\beta\vert=k}\bar\Phi_{\beta}f(\bar\by_\beta)L_{\beta,\by}(\bx).
\end{aligned}\end{equation}

Note that it follows from the definition that,
\begin{equation}\label{C1}
\bar\Phi_{\beta+e_1} f(z,\bar\by_{\beta})=\bar\Phi_{\beta+e_1} f(\bar\by_{\beta+e_1})+(z-y_1)\bar\Phi_{\beta+2e_1}f(z,\bar\by_{\beta+e_1}).
\end{equation}

Let us denote $f_a^i(\bx):=\bar\Phi_{e_i}f(a,\bx^{(i)})$, where $(a,\bx^{(i)}):=(x_1,\cdots,x_{i-1},a,x_i,x_{i+1},\cdots,x_d)$ for $\bx\in B^d$, $a\in B$, and $i=1,\cdots,d$. Also, for $\bx_{\beta}=(\bx_1,\cdots,\bx_d)\in B^{\beta}$, let us denote $(a,\bx_{\beta}^{(i)}):= (\bx_1,\cdots,\bx_{i-1}, a,\bx_i,\bx_{i+1},\cdots,\bx_d)$, and $(a,b,\bx_{\beta}^{(i)}):= (\bx_1,\cdots,\bx_{i-1}, a,b,\bx_i,\bx_{i+1},\cdots,\bx_d).$

Therefore note by definition, 
$$
\bar\Phi_{\alpha+e_i}f(y_i, \bx_{\alpha(1)}^{(i)})=\bar\Phi_\alpha f_{y_i}^i(\bx_{\alpha(1)}).
$$

 \begin{lemma}\label{chain rule}Suppose that $z\in B,\by\in B^d$. Let $f: B^d\subset B(0,1)^d\to \Q_{\nu}$ a $C^{l+1}$ map.  For any multi-index $\alpha$ with $\vert \alpha\vert \leq l-1$ and for any $\bx_{\alpha(1)}\in B^{\alpha(1)}$ and $i=1,\dots,d$ we have,
 \begin{equation}\label{chain}
\begin{aligned}
  &\ \ \ \bar\Phi_{\alpha+e_i} P_{f,\by,l}(z,\bx_{\alpha(1)}^{(i)})\\&
  = (z-y_i)\bar\Phi_{\alpha+e_i} P_{f^i_{y_i},\by,l-1} (z,\bx_{\alpha(1)}^{(i)})+ \bar\Phi_{\alpha}P_{f^i_{y_i},\by,l-1} (\bx_{\alpha(1)}).\end{aligned}
  \end{equation}
 \end{lemma}
 \begin{proof}

Let us note that $$
\begin{aligned}
  &\ \ \ \bar\Phi_{\alpha+e_i} P_{f,\by,l}(z,\bx_{\alpha(1)}^{(i)})\\
  & =   \sum_{\vert \beta\vert =0} \bar\Phi_{\beta}f(\bar\by_{\beta})\bar\Phi_{\alpha+e_i}L_{\beta,\by}(z,\bx_{\alpha(1)}^{(i)})+ \cdots+\sum_{\vert\beta\vert =l}\bar\Phi_{\beta} f(\bar\by_{\beta})\bar\Phi_{\alpha+e_i}L_{\beta,\by}(z,\bx_{\alpha(1)}^{(i)})\\
  & =   \sum_{\vert \beta\vert =0} \bar\Phi_{\beta+e_i}f(\bar\by_{\beta+e_i})\bar\Phi_{\alpha+e_i}L_{\beta+e_i,\by}(z,\bx_{\alpha(1)}^{(i)})+ \cdots+\sum_{\vert\beta\vert =l-1}\bar\Phi_{\beta+e_i} f(\bar\by_{\beta+e_i})\bar\Phi_{\alpha+e_i}L_{\beta+e_i,\by}(z,\bx_{\alpha(1)}^{(i)})\\
  &= \bar\Phi_{\alpha+e_i}\left( L_{e_i,\by} \times\left(\sum_{\vert \beta\vert =0} \bar\Phi_{\beta+e_i}f(\bar\by_{\beta+e_i})L_{\beta,\by}+ \cdots+\sum_{\vert\beta\vert =l-1}\bar\Phi_{\beta+e_i} f(\bar\by_{\beta+e_i})L_{\beta,\by}\right)\right)(z,\bx_{\alpha(1)}^{(i)})\\
  &= \bar\Phi_{\alpha+e_i}\left( L_{e_i,\by}\times\left(\sum_{\vert \beta\vert =0} \bar\Phi_{\beta}f^i_{y_i}(\bar\by_{\beta})L_{\beta,\by}+ \cdots+\sum_{\vert\beta\vert =l-1}\bar\Phi_{\beta} f_{y_i}^i(\bar\by_{\beta})L_{\beta,\by}\right)\right)(z,\bx_{\alpha(1)}^{(i)})\\
  &= \bar\Phi_{\alpha+e_i}\left( L_{e_i,\by} \times\left(P_{f^i_{y_i},\by,l-1}\right)\right) \left(z,\bx_{\alpha(1)}^{(i)}\right)\\
  &= L_{e_i,\by}(z)\bar\Phi_{\alpha+e_i}\left( P_{f^i_{y_i},\by,l-1}\right)(z,\bx_{\alpha(1)}^{(i)})+ \bar\Phi_{\alpha}\left(P_{f^i_{y_i},\by,l-1}\right) (\bx_{\alpha(1)}).
  \end{aligned}$$
{In the equations above, by $\times$ we mean multiplication of functions.}
The last line follows from \cref{chain chain}, which is a version of the chain rule. 
 \end{proof}

 \begin{theorem}\label{taylor}
Suppose that $f:B(\by,r)\subset B(0,1)^d\to \Q_\nu$ be a $C^{l+1}$ map, then for any $0\leq \vert\eta\vert \leq l$,
\begin{equation}\label{star}
\begin{aligned}
&\sup_{\bx_{\eta(1)}\in B(\by,r)^{\eta(1)}} \vert \bar\Phi_{\eta} f(\bx_{\eta(1)})-\bar \Phi_{\eta} P_{f, \by, l}(\bx_{\eta(1)})\vert \\
& \leq r^{l-\vert \eta\vert} \sup_{\substack{ \bw_1,\bw_2\in B(\by,r)^{\beta(1)},\\
\vert\beta\vert=l}}
\vert \bar\Phi_\beta f(\bw_1)-\bar \Phi_\beta f(\bw_2)\vert.
\end{aligned}
\end{equation}
\end{theorem}
\begin{proof}
 We proceed by induction on $l$. If $l=0$, then $\vert\eta\vert=0$ and $P_{f,\by,0}(\bx)=f(\by)$, hence the statement of the theorem follows for $l=0$. Assume the the statement is true for $l-1$. We examine $\eta$ a multi index with $\vert \eta\vert\leq l$ in three cases, $\vert\eta\vert=0, l$ and $1\leq \vert \eta\vert \leq l-1.$ Let $l-1\geq k\geq 1$ and  fix any $1\leq i\leq d$. Take $\alpha$ be a multi index such that $\vert\alpha\vert =k-1.$ Let 
 $$\mathbf{M}= \sup_{\substack{ \bw_1,\bw_2\in B(\by,r)^{\beta(1)},\\
\vert\beta\vert=l}}
\vert \bar\Phi_\beta f(\bw_1)-\bar \Phi_\beta f(\bw_2)\vert.$$
 By the induction hypothesis, we have for any $\bx_{\alpha(1)}\in B(\by,r)^{\alpha(1)}$,
 \begin{equation}\label{hyp1}
     \vert \bar\Phi_{\alpha}P_{f^i_{y_i},\by,l-1} (\bx_{\alpha(1)})-\bar\Phi_\alpha f^i_{y_i}(\bx_{\alpha(1)})\vert \leq r^{l-k}\mathbf{M}.
 \end{equation}

 Also using induction we have for any $(z,\bx_{\alpha(1)}^{(i)})\in B(\by,r)^{\alpha+e_i}$,
 \begin{equation}\label{hyp2}
     \vert \bar\Phi_{\alpha+e_i} P_{f^i_{y_i},\by,l-1} (z,\bx_{\alpha(1)}^{(i)})- \bar\Phi_{\alpha+e_i} f^i_{y_i} (z,\bx_{\alpha(1)}^{(i)})\vert\leq r^{l-1-k}\mathbf{M}.
 \end{equation}
  On the right hand side of \cref{hyp1} and \cref{hyp2} $\mathbf{M}$ appears because for any multi index $\beta$ with $\vert \beta\vert=l-1, $ $\bar\Phi_{\beta} f_{y_i}^i(\bw)= \bar \Phi_{\beta+e_i}f(y_i,\bw^{(i)}).$

 Note that $$\begin{aligned}\bar\Phi_{\alpha+e_i}f(z,\bx_{\alpha(1)}^{(i)})&=(z-y_i)\bar\Phi_{\alpha+2e_i}f(z,y_i,\bx_{\alpha(1)}^{(i)})+\bar\Phi_{\alpha+e_i}f(y_i,\bx_{\alpha(1)}^{(i)})\\
 & =(z-y_i)\bar\Phi_{\alpha+e_i}f_{y_i}^i(z,\bx_{\alpha(1)}^{(i)})+\bar\Phi_{\alpha}f^i_{y_i}(\bx_{\alpha(1)}).\end{aligned}$$
 In the last line of the above equality, we use the definition of the function $f_{y_i}^i(\bx)=\bar\Phi_{e_i}f(y_i,\bx^{(i)}).$
 Therefore, by \cref{chain rule} we have the following:
 \begin{equation}\label{induction1}\begin{aligned}
   \vert\bar\Phi_{\alpha+e_i} P_{f,\by,l}(z,\bx_{\alpha(1)}^{(i)})-\bar\Phi_{\alpha+e_i}f(z,\bx_{\alpha(1)}^{(i)})\vert \leq r^{l-k}\mathbf{M}.
 \end{aligned}\end{equation}
 Now let us take a multi-index $\alpha$, such that  $\vert \alpha\vert = l-1$, then we have $$\bar\Phi_{\alpha+e_i}P_{f,\by,l}(z,\bx_{\alpha(1)}^{(i)})= \bar\Phi_{\alpha+e_i}f(\bar\by_{\alpha+e_i}).$$
 Therefore for any $(z,\bx_{\alpha(1)}^{(i)})\in B(\by,r)^{\alpha+e_i}$,
 \begin{equation}\label{induction2}
  \vert\bar\Phi_{\alpha+e_i}f(\bar\by_{\alpha+e_i})-\bar\Phi_{\alpha+e_i}f(z,\bx_{\alpha(1)})\vert\leq r^{0}\mathbf{M}.
 \end{equation}
 Now suppose $k=0$, then by \cref{TAYLOR} we have,
 $$
 \vert f(\bx)-P_{f,\by,l}(\bx)\vert  \leq \max_{\substack{\vert\beta\vert=l+1,\\ \beta=(0,\cdots,0,i_j,\cdots,i_d),\\i_j>0\\j=1,\cdots,d}}\vert \bar\Phi_{\beta}f(x_1,\cdots,x_j,y_j^{i_j},y_{j+1}^{i_{j+1}+1},\cdots,y_d^{i_d+1}) L_{\beta,\by}(\bx)\vert
 $$
Note that for $\beta=(0,\cdots,0,i_j,\cdots,i_d), i_j>0$ above, $$L_{\beta,\by}(\bx)=\prod_{s=j}^d(x_s-y_s)^{i_s}=(x_{j}-y_{j})L_{\beta-e_j,\by}(\bx).$$  Since $\bx\in B(\by,r),$ and $\vert \beta-e_j\vert=l$, the right hand side of the above inequality can be bounded above by the following,
 $$\begin{aligned}
 &\leq r^{l} \max_{\substack{\vert\beta\vert=l+1,\\ \beta=(0,\cdots,0,i_j,\cdots,i_d),\\i_j>0\\j=1,\cdots,d}}\vert \bar\Phi_{\beta}f(x_1,\cdots,x_j,y_j^{i_j},y_{j+1}^{i_{j+1}+1},\cdots,y_d^{i_d+1})(x_j-y_j)\vert.
 \end{aligned}
 $$

 Since for any multi-index $
 \beta$, $$\begin{aligned}
 &\bar\Phi_{\beta}f(x_1,\cdots, x_j,y_j^{i_j},y_{j+1}^{i_{j+1}+1},\cdots,y_d^{i_d+1})(x_j-y_j)\\
 &=\bar\Phi_{\beta-e_j}f(x_1,\cdots,x_j,y_j^{i_j-1},y_{j+1}^{i_{j+1}+1},\cdots,y_d^{i_d+1})-\bar\Phi_{\beta-e_j}f(x_1,\cdots,x_{j-1},,y_j^{i_j},y_{j+1}^{i_{j+1}+1},\cdots,y_d^{i_d+1}),\end{aligned}$$
 we conclude that for any $\bx\in B(\by,r)$, we have that
 \begin{equation}\label{induction3}
     \vert f(\bx)-P_{f,\by,l}(\bx)\vert \leq r^{l}\mathbf{M}.\end{equation}
 Hence \cref{induction1}, \cref{induction2} and \cref{induction3} yield the theorem.

\end{proof}

  For $\mathbf{f}=(f_1,\cdots,f_n)$
we define $$ \mathcal{S}_{\f}=\{c_0+c_1f_1+\cdots+c_nf_n~|~\max_{i=0}^n \vert c_i\vert=1\},$$ where $f_i:B\subset B(0,1)^d\to \Q_\nu$ for all $i=1,\cdots,n$.

\begin{lemma}\label{equi}
Suppose that $\f=(f_1,f_2,\cdots,f_n):B\subset B(0,1)^d\to\Q_\nu^n$ is a $C^{l+1}$ map, and $\bx_0\in B$. Then for any $\varepsilon>0$ there exists a neighbourhood $V\subset B$ of $\bx_0$ such that for any ball $ B(\by,r)\subset V$, for any multi-index $\beta$, $0\leq \vert\beta\vert\leq l$, and for any $g\in \mathcal{S}_{\f}$,
we have \begin{equation}
    \sup_{\bx_{\beta(1)}\in B(\by,r)^{\beta(1)}} \vert \bar\Phi_{\beta} g(\bx_{\beta(1)})-\bar \Phi_{\beta} P_{g, \by, l}(\bx_{\beta(1)})\vert<\varepsilon r^{l-\vert \beta\vert}.\end{equation}
\end{lemma}
\begin{proof}
Using \cref{taylor} it is enough to show that for any $\varepsilon>0$, there exists a neighbourhood $V$ such that for any $B(\by,r)\subset V$, we have $$ \sup_{\substack{\bw_1,\bw_2\in B(\by,r)^{\beta(1)},\\
\vert \beta\vert=l}} \vert \bar\Phi_{\beta} g(\bw_1)-\bar \Phi_{\beta} g(\bw_2)\vert<\varepsilon $$ for any $g\in \mathcal{S}_{\f}$. Since the family $\{\bar\Phi_{\beta} g ~|~\vert \beta\vert=l, g\in \mathcal{S}_{\f} \}$ is equicontinuous, for any $\varepsilon>0$, we can always guarantee a neigbourhood $V(\varepsilon)$ such that $\vert \bar\Phi_{\beta}g(\bw_1)-\bar\Phi_{\beta}g(\bw_2)\vert<\varepsilon$ for all $\beta$ with $\vert \beta\vert=l$, $B(\by,r)\subset V(\varepsilon)$, $\bw_1,\bw_2\in B(\by,r)^{\beta(1)}$ and $g\in \mathcal{S}_{\f}$.
\end{proof}
\section{Nondegeneracy and normalization}

\begin{lemma}\label{bounds}
Suppose that $\f=(f_1,f_2,\cdots,f_n):B\subset B(0,1)^d\to\Q_\nu^n$ is a $C^{l+1}$ map, which is $l$-nondegenerate at $\bx_0$. Then there is an open set $\bx_0\in V_0\subset B$ and $\eta>0$ such that for any $B(\by,r)\subset V_0$ , and $g\in \mathcal{S}_{\f}$,
\begin{equation}
    \Vert g\Vert_{B(\by,r)}\geq \eta r^l
\end{equation} and \begin{equation}
    \ \frac{\Vert P_{g,\by,l}\Vert_{B(\by,r)}}{\Vert g\Vert_{B(\by,r)}}=1.
\end{equation}

\end{lemma}
\begin{proof}
Let $V_0$ be an open ball such that $\f=(f_1,\cdots,f_n)$ is nondegenerate for any $\by\in V_0$. Let us rescale $P_{g,\by,l}$  as $$ Q_{g,\by}(\bx):=P_{g,\by,l}(\bx+\by), \bx\in \Q_\nu^d,$$ where $g\in \mathcal{S}_{\f}$. Let us equip $\mathcal{P}_{d,l}$ with the norm $\Vert \cdot\Vert_{d,l}$ that is the maximum of all the $\nu$-norms of all coefficients. By nondegeneracy of $f$ and compactness of $\mathcal{S}_{\f}$ and $V_0$, we have that the set $\{ Q_{g,\by}~:~g\in\mathcal{S}_{\f}, \by\in V_0\}$ is bounded away from $0$, which means for all $g\in \mathcal{S}_{\f}$, $\by\in V_0$, $\Vert Q_{g,\by}\Vert_{d,l}\geq \eta_1$ for some $\eta_1>0$. Moreover, we can  rescale again, for each $\by\in V_0$ and $r=\nu^{-R}>0$ such that $B(\by,r)\subset V_0$, we define $$
Q_{g,\by}^{R}(\bx):= \nu^{-lR} Q_{g,\by}(\nu^{R} \bx).$$  Since $\{ Q_{g,\by}~:~g\in\mathcal{S}_{\f}, \by\in V_0\}$ is bounded away from $0$, we will have $\{ Q_{g,\by}^{R}~:~g\in\mathcal{S}_{\f}, B(\by,\nu^{-R})\subset V_0\}$ is also bounded away from $0$. The reason is $\Vert Q_{g,\by}^{R}\Vert_{d,l} \geq  \Vert Q_{g,\by}\Vert_{d,l},$ since norm of $\beta$-th coefficient in $Q_{g,\by}^{R}$ is greater than norm of $\beta$-th coefficient in $Q_{g,\by}$ for every multi-index $\beta, \vert \beta\vert\leq l.$ 
By comparing with the norm $P\to \Vert P\Vert_{B(\mathbf{0},1)},$ the above implies that for all $g\in \mathcal{S}_{\f}$, $ \Vert Q_{g,\by}^{R}\Vert_{B(0,1)^d}\geq \eta$ for some $\eta>0$. This gives us for all $B(\by,r)\subset V_0$ and $g\in\mathcal{S}_{\f}$.
$$
 \Vert P_{g,\by,l}\Vert_{B(\by,r)}= \nu^{-lR} \Vert Q_{g,\by}^{R}\Vert_{B(0,1)^d} \geq r^l\eta.$$ Now using \cref{equi}, for $\vert\beta\vert=0$, for sufficiently small $V_0$, $$
 \sup_{\bx\in B(\by,r)}\vert g(\bx)-P_{g,\by,l}(\bx)\vert\leq \frac{1}{\nu} \nu^{-lR}\eta=\frac{1}{\nu}r^l\eta.$$ Hence, we have
 \begin{equation}\label{k3}
     \Vert g\Vert_{B(\by,r)}\geq  r^{l}\eta.
     \end{equation}
 Note,

$$\Vert P_{g,\by,l}\Vert_{B(\by,r)}\leq \max(\frac{1}{\nu}r_{\nu}^l\eta, \Vert g\Vert_{B(\by,r)}).
$$

 \noindent Hence by \cref{k3} we have, $$\Vert P_{g,\by,l}\Vert_{B(\by,r)}\leq \Vert g\Vert_{B(\by,r)}.$$ Similarly we get, $$\Vert g\Vert_{B(\by,r)} \leq \Vert P_{g,\by,l}\Vert_{B(\by,r)}.$$
 Therefore,
$$
  \frac{\Vert P_{g,\by,l}\Vert_{B(\by,r)}}{\Vert g\Vert_{B(\by,r)}}=1.
$$

\end{proof}

Similar to the way we have normalized polynomials to make them functions in the unit ball $B(0,1)^d$, we will normalize functions in $\mathcal{S}_{\f}$. So for $V_0$ a sufficiently small neighbourhood of $\bx_0$, $B(\by,r)\subset V_0$ and $g\in\mathcal{S}_{\f}$, and $r>0$ is a power of $\nu,$ we define the function $$
g_{r,\by,g}(\bx):= \Vert g\Vert_{B(\by,r)} g(\by+r^{-1}\bx),$$ which is defined on $B(0,1)^d$. We note that the above normalization and shift in the definition of $g_{r,\by,g}$ are taken such that using $\nu$-adic norm, $\Vert g_{r,\by,g}\Vert_{B(0,1)^d}=1$. We also consider the collection
$$
\mathcal{G}(\f,V):=\{g_{r,\by,g}~|~g\in \mathcal{S}_{\f}, B(\by,r)\subset V\}.
$$
\begin{lemma}\label{LL}
For any $\varepsilon>0$ there exists a neighbourhood $V\subset B$ of $\bx_0$ such that for any $\phi\in \mathcal{G}(\f,V)$ one has
\begin{equation}
    \max_{\vert\beta\vert\leq l+1} \sup_{\substack{\bx_{\beta(1)}\in B(\mathbf{0},1)^{\beta(1)}}}\vert (\bar \Phi_{\beta}\phi-\bar\Phi_{\beta} P_{\phi,\mathbf{0},l})(\bx_{\beta(1)})\vert<\varepsilon.
\end{equation}
\end{lemma}
\begin{proof}
Let us choose a small $V$ that we get from \cref{equi} and \cref{bounds}. Let us denote $y_j+r^{-1}\bx_j:=y_j^{i_j}+r^{-1}(x_{j,1},\cdots,x_{j,i_j+1})$, and $\by+r^{-1}\bx_{\beta(1)}=(y_1+r^{-1}\bx_1,\cdots, y_d+r^{-1}\bx_d)$
Note that if $\phi=g_{r,\by,g}\in \mathcal{G}(\f,V)$,then  $$\bar\Phi_{\beta} \phi(\bx_{\beta(1)})=\frac{\Vert g\Vert_{B(\by,r)}}{r^{\vert\beta\vert}}\bar\Phi_{\beta} g(\by+r^{-1}\bx_{\beta(1)})$$ and $$\bar\Phi_{\beta} P_{\phi,\mathbf{0},l}(\bx_{\beta(1)})=\frac{ \Vert g\Vert_{B(\by,r)}}{r^{\vert\beta\vert}}\bar\Phi_{\beta} P_{g,\by,l}(\by+r^{-1}\bx_{\beta(1)}).$$  Hence, we have

$$
\begin{aligned}
  &\vert (\bar \Phi_{\beta}\phi-\bar\Phi_{\beta} P_{\phi,\mathbf{0},l})(\bx_{\beta(1)})\vert
  \\&= \frac{r^{\vert \beta\vert}}{\Vert g\Vert_{B(\by,r)}}\vert \bar\Phi_{\beta} g(\by+r^{-1}\bx_{\beta(1)})-\bar\Phi^k P_{g,\by,l}(\by+r^{-1}\bx_{\beta(1)})\vert.
\end{aligned} $$
Using \cref{equi} and \cref{bounds} for $\vert\beta\vert\leq l$ we  get,
$$ \vert (\bar \Phi_{\beta}\phi-\bar\Phi_{\beta} P_{\phi,\mathbf{0},l})(\bx_{\beta(1)})\vert\leq \frac{r^{\vert\beta\vert}}{\eta r^l}{\varepsilon \eta r^{l-\vert\beta\vert}}=\varepsilon.$$

For $\vert\beta\vert=l+1$ by compactness, we have a $K>0$ such that $\Vert\bar\Phi_{\beta}g\Vert_{V}\leq K$ for all $g\in \mathcal{S}_{\f}$,
$$\begin{aligned}
\Vert \bar\Phi_{\beta}\phi-\bar\Phi_{\beta}P_{\phi,\mathbf{0},l}\Vert_{B(0,1)^{l+1+d}}=&\Vert \bar\Phi_{\beta}\phi\Vert_{B(0,1)^{l+1+d}}\\=&\frac{r^{l+1}}{\Vert g\Vert_{B(\by,r)}}\sup_{\bx_{\beta(1)}\in B(0,1)^{l+1+d}}\vert\bar\Phi_{\beta} g(\by+r^{-1}\bx_{\beta(1)})\vert\leq \frac{r}{\eta}K<\varepsilon,\end{aligned}$$ if $V$ is small enough.
\end{proof}
\section{Completing the proof of \cref{mainp}}

\begin{theorem}
Let $\mathbf{B}\subset B(0,1)^d\subset \Q_\nu$ be an open ball and let $\f:\mathbf{B}\to\Q_\nu^n$ be a $C^{l+1}$ map which is $l$ nondegenerate at $\bx_0\in \mathbf{B}$. Let $\mu$ be a measure which is $D$-Federer and absolutely $(C,\alpha)$-decaying on $\mathbf{B}$ for some $D,C,\alpha>0$. Then there exists a neighborhood $V\subset \mathbf{B}$ of $\bx_0$ and a positive $\tilde {C}$ such that for any $g\in \mathcal{S}_{\f}$ is absolutely $(\tilde{C},\alpha')$-good on $V$ with respect to $\mu$.
\end{theorem}
\begin{proof}
 By \cref{good}, it suffices to find $C_{\alpha_1,\alpha_1'},s,S$ for multi-index $\alpha_1,\alpha_1'$ and $V\subset \mathbf{B}$ such that for any $\phi\in \mathcal{G}(\f,V) $ there is a multi-index $\beta$ with $\vert\beta\vert=k \in\N$ such that, for every  $\alpha_1+\alpha_2=\beta=\alpha_1'+\alpha_2'$,  $\vert \alpha_1\vert >\vert \alpha_1'\vert$,
$$
\inf_{B(0,1)^{\vert\alpha_1\vert+d}} \vert\bar\Phi_{\alpha_1} \partial_{\alpha_2}\phi \vert \leq C_{\alpha_1,\alpha_1'}\inf_{B(0,1)^{ \vert\alpha_1'\vert+d}}\vert \bar\Phi_{\alpha_1'}\partial_{\alpha_2'}\phi\vert
$$
and
$$
\inf_{B(0,1)^{\vert\beta\vert+d }} \vert\bar\Phi_{\beta} \phi\vert>s, $$ and for each $\vert\eta\vert\leq k+1$,
$$
\Vert\bar\Phi_{\eta} \phi\Vert_{{B(0,1)}^{\vert\eta\vert+d}}<S.
$$

Using \cref{bounds}, we can choose $V_0$ around $\bx_0$ such that for $B(\by,r)\subset V_0$, $g\in \mathcal{S}_{\f},$
$$ \Vert P_{g,\by,l}\Vert_{B(\by,r)}=\Vert g\Vert_{B(\by,r)}.$$ This implies

$$\Vert P_{\phi,\mathbf{0},l}\Vert_{B(\mathbf{0},1)}=\frac{1}{\Vert g\Vert_{B(\by,r)}}\sup_{\bx\in B(\mathbf{0},1)}\vert P_{g, \by,l}(\by+r^{-1}\bx)\vert=1.$$
Now using \cref{poly} we can find $\vert\beta\vert=k\leq l$ and $s,S, C_{\alpha_1,\alpha'_1}>0$ such  that, for every  $\alpha_1+\alpha_2=\beta=\alpha_1'+\alpha_2'$,  $\vert \alpha_1\vert >\vert \alpha_1'\vert$,
 \begin{equation}\label{extra}\inf_{B(0,1)^{\vert \alpha_1\vert+d}} \vert\bar\Phi_{\alpha_1} \partial_{\alpha_2}P_{\phi,\mathbf{0},l}\vert \leq C_{\alpha_1,\alpha_1'}\inf_{B(0,1)^{\vert \alpha_1'\vert+d}} \vert \bar\Phi_{\alpha_1'}\partial_{\alpha_2'}P_{\phi,\mathbf{0},l}\vert,\end{equation} and
 \begin{equation}\label{down}
 \inf_{B(0,1)^{k+d}}\vert \bar \Phi_{\beta} P_{\phi,\mathbf{0},l}\vert > s
 \end{equation}
 and \begin{equation}\label{up}
\Vert \bar \Phi_{\eta} P_{\phi,\mathbf{0},l}\Vert_{B(0,1)^{\vert\eta\vert+d}}<S  ~ ~\forall~~ \vert\eta\vert\geq 0.\end{equation}
Note that from \cref{extra} for $\beta=\alpha_1'+\alpha_2', \vert \beta\vert>\vert\alpha_1'\vert$ we have,
$$\inf_{B(0,1)^{k+d}} \vert\bar\Phi_{\beta} P_{\phi,\mathbf{0},l}\vert \leq C_{\beta,\alpha_1'}\inf_{B(0,1)^{\vert \alpha_1'\vert+d}} \vert \bar\Phi_{\alpha_1'}\partial_{\alpha_2'}P_{\phi,\mathbf{0},l}\vert.$$

Let us take $a\in\N$ such that $\frac{1}{\nu^a}<\max_{\alpha_1+\alpha_2=\beta=\alpha_1'+\alpha_2'} C^{-1}_{\alpha_1,\alpha_1'}$. Hence by \cref{down} we have that
\begin{equation}\label{end1}
\inf_{B(0,1)^{\vert \alpha_1'\vert+d}} \vert \bar\Phi_{\alpha_1'}\partial_{\alpha_2'}P_{\phi,\mathbf{0},l}\vert>s C^{-1}_{\beta,\alpha_1'}>\frac{s}{\nu^a}.\end{equation}
Let us take $0<\varepsilon<\frac{s}{\nu^a}$. If $\alpha_1'+\alpha_2'=\beta$, then for any $\bz'\in B(0,1)^{\vert\alpha_1'\vert+d} $, by \cref{verybasic}, $$\bar\Phi_{\alpha_1'}\partial_{\alpha_2'}P_{\phi,\mathbf{0},l}(\bz')=\sum_{\bt_{\beta}} \bar\Phi_{\beta}P_{\phi,\mathbf{0},l}(\bt_\beta), \text{ and } \bar\Phi_{\alpha_1'}\partial_{\alpha_2'}\phi(\bz')=\sum_{\bt_{\beta}} \bar\Phi_{\beta}\phi(\bt_\beta),$$ where $\bt_{\beta}$ has coordinates in terms of $\bz'$. It should be clear that applying \cref{verybasic} multiple times will give  $\bt_\beta$'s, whose coordinates will be a variation of coordinates of $\bz'$. Therefore by \cref{LL} and the equality above, and property of ultrametric norm, we have \begin{equation}\label{end2}
\Vert \bar\Phi_{\alpha_1'}\partial_{\alpha_2'}\phi-\bar\Phi_{\alpha_1'}\partial_{\alpha_2'}P_{\phi,\mathbf{0},l}\Vert_{B(0,1)^{\vert\alpha_1'\vert+d}}< \varepsilon.\end{equation}
Since $\varepsilon<\frac{s}{\nu^a}$, by \cref{end1} and \cref{end2} we have that for any $\bz'\in B(0,1)^{\vert\alpha_1'\vert+d} $,
\begin{equation}\label{end3}
\vert \bar\Phi_{\alpha_1'}\partial_{\alpha_2'}\phi(\bz')\vert= \vert\bar\Phi_{\alpha_1'}\partial_{\alpha_2'}P_{\phi,\mathbf{0},l}(\bz')\vert
\end{equation} for any $\alpha_1'+\alpha_2'=\beta$. Hence by \cref{extra} we have, for every  $\alpha_1+\alpha_2=\beta=\alpha_1'+\alpha_2'$,  $\vert \alpha_1\vert >\vert \alpha_1'\vert$,  $$\inf_{B(0,1)^{\vert \alpha_1\vert+d}} \vert\bar\Phi_{\alpha_1} \partial_{\alpha_2}\phi\vert \leq C_{\alpha_1,\alpha_1'}\inf_{B(0,1)^{\vert \alpha_1'\vert+d}} \vert \bar\Phi_{\alpha_1'}\partial_{\alpha_2'}\phi\vert.$$

By \cref{LL} and \cref{down}, we get that for any $\bx_{\beta(1)}\in B(\mathbf{0},1)^{\beta(1)}$, $$
s< \max({\vert \bar \Phi_{\beta} \phi(\bx_{\beta(1)})\vert , \varepsilon}).$$ Since $\varepsilon<\frac{s}{\nu^a}$ we get, $$
\inf_{B(0,1)^{k+d}} \vert \bar\Phi_{\beta}\phi\vert > s .$$
Now for upper bound, we apply \cref{LL} and use \cref{up} to get

$$
\Vert \bar \Phi_{\eta}\phi\Vert_{B(0,1)^{k+d}} \leq \max(\varepsilon, S)=S$$ for every $\vert\eta\vert\leq l,$ since $\varepsilon<\frac{s}{\nu^a}<S.$

\end{proof}


\section{Appendix}\label{appen}

 In this appendix, we will prove the following proposition which can be thought of as the analogue of Taylor's theorem for real variables. One of the main issues here is the careful computation of the `error' term as we need it for \cref{taylor}. In the one dimensional case, the proof is much easier, and is well known \cite[Theorem $29.4$]{Sc}. We could not find a reference in general and so we provide a proof. Let us recall that for a $C^k$ function $f:\Q_\nu^d\to \Q_\nu$, we denoted by $P_{f,\by,k}$ the Taylor polynomial of $f$ at $\by\in\Q_\nu^d$; see \eqref{defn:taylor}.
\begin{proposition}\label{TAYLOR} For a $C^{l+1}$ map $f:\Q_\nu^d\to\Q_\nu$ and $\by\in \Q_\nu^d$, we have the following
 for $k\leq l$:
 \begin{equation}\label{taylorp}
 \begin{aligned}f(\bx) &=P_{f,\by,k}(\bx)+\sum_{\substack{{\vert \beta\vert =k+1}\\\beta=(i_1,\cdots,i_d)\\i_1>0}}\bar\Phi_{\beta}f(x_1,y_1^{i_1},y_2^{i_2+1},\cdots,y_d^{i_d+1})L_{\beta,\by}(\bx)\\
 &+ \sum_{\substack{{\vert \beta\vert =k+1}\\ \beta=(0,i_2,\cdots,i_d)\\i_2>0}}\bar\Phi_{\beta}f(x_1,x_2,y_2^{i_2},y_3^{i_3+1},\cdots,y_d^{i_d+1})L_{\beta,\by}(\bx)\\
 & + \cdots+\bar\Phi_{(k+1)e_d}f(x_1,\cdots,x_d,y_d^{k+1})L_{(k+1)e_d,\by}(\bx)\end{aligned}.\end{equation}
\end{proposition}
\begin{remark} We remark that Proposition 11.1 differs from the Lagrange remainder formula and also from the Mean Value theorem. Let us explain the simplest meaningful case to point out the difference. Let $g:\Q_\nu\to \Q_\nu$ be a $C^2$ map. \cref{TAYLOR} shows
\begin{equation}\label{lagrange_p}
g(x)=g(y)+\underbrace{\bar\Phi^1 g(y,y)}_{g'(y)}(x-y)+\bar\Phi^2g(x,y,y)(x-y)^2.
\end{equation}
Note that the remainder term contains the function $\bar\Phi^2g$ evaluated at $(x,y,y)$ and  the function $\bar\Phi^2g$ has a larger domain than that of $g''$, and  $\frac{1}{2}g''(z)=\bar\Phi^2g(z,z,z)$ for all $z\in \Q_\nu.$

On the other hand for a function $\psi:\R\to \R$ a $C^2$ map, by Lagrange form, we have 
$$
\psi(x)=\psi(y)+\psi'(y)(x-y)+\frac{1}{2}\psi''(c) (x-y)^2,
$$
where $c$ is a point between $x$ and $y$ in $\R.$ \cref{TAYLOR} is the generalization of \eqref{lagrange_p} in higher dimensions.
\end{remark}

\begin{proof}

 We proceed by induction on $k$ and accordingly, set $k=0$. Then,
 $$\begin{aligned} & f(\bx)-f(\by)\\
 & =\sum_{i=1}^{d}f(x_1,\cdots,x_{i-1},y_{i},y_{i+1},\cdots,y_{d})-f(x_1,\cdots,x_{i-1},x_{i},y_{i+1},\cdots,y_{d})\\
 &= \sum_{i=1}^d \bar\Phi_{e_i}f(x_1,\cdots,x_{i-1},x_{i},y_{i},y_{i+1},\cdots,y_d)L_{e_j,\by}(\bx).\end{aligned}$$ Since $P_{f,\by,0}(\bx)=f(\by)$, we have the statement to be true for $k=0$. By induction hypothesis let us assume that the statement is true for $k=m$:
 \begin{equation}\label{hyp1+}
 \begin{aligned} f(\bx)=P_{f,\by,m}(\bx)+I_{1}+\cdots+I_d,\\
 \end{aligned}\end{equation} where $$
 I_j:=\sum_{\substack{{\vert \beta\vert =m+1}\\ \beta=(0,\cdots,0,i_j,\cdots,i_d)\\i_j>0}}\bar\Phi_{\beta}f(x_1,\cdots,x_j,y_j^{i_j},y_{j+1}^{i_{j+1}+1},\cdots,y_d^{i_d+1})L_{\beta,\by}(\bx).$$ We want to show that the statement is true for $k=m+1\leq l$.

 We know by definition, that the coefficient of a term in the summand $I_1$ is,  $$\begin{aligned}
 \bar\Phi_{\beta}f(x_1,y_1^{i_1},y_2^{i_2+1},\cdots,y_d^{i_d+1})=& \bar\Phi_{\beta}f(y_1^{i_1+1},y_2^{i_2+1},\cdots,y_d^{i_d+1})\\& +\bar\Phi_{\beta+e_1}f(x_1,y_1^{i_1+1},y_2^{i_2+1},\cdots,y_d^{i_d+1})L_{e_1,\by}(\bx).\end{aligned}$$ We also have, that the coefficient of a term in the summand $I_2$ is,
 $$\begin{aligned}
 & \bar\Phi_{\beta}f(x_1,x_2,y_2^{i_2},y_3^{i_3+1},\cdots,y_d^{i_d+1})\\= & \bar\Phi_{\beta}f(x_1,y_2^{i_2+1},y_3^{i_3+1},\cdots,y_d^{i_d+1})+ \bar\Phi_{\beta+e_2}f(x_1,x_2,y_2^{i_2+1},y_3^{i_3+1},\cdots,y_d^{i_d+1})L_{e_2,\by}(\bx)\\
 = & \bar\Phi_{\beta}f(y_1,y_2^{i_2+1},y_3^{i_3+1},\cdots,y_d^{i_d+1})+ \bar\Phi_{\beta+e_2}f(x_1,x_2,y_2^{i_2+1},y_3^{i_3+1},\cdots,y_d^{i_d+1})L_{e_2,\by}(\bx)\\ & + \bar\Phi_{\beta+e_1}f(x_1,y_1,y_2^{i_2+1},y_3^{i_3+1},\cdots,y_d^{i_d+1})L_{e_1,\by}(\bx).
 \end{aligned}$$
 Continuing like this we get that the coefficient of a term in $I_d$ is,
 $$\begin{aligned}
 & \bar\Phi_{(m+1)e_d}f(x_1,\cdots,x_d,y_d^{m+1})\\
 & =\bar\Phi_{(m+1)e_d}f(x_1,\cdots, x_{d-1},y_d^{m+2})+\bar\Phi_{(m+2)e_d}f(x_1,\cdots,x_d,y_d^{m+2})L_{e_d,\by}(\bx)\\
 &=\bar\Phi_{(m+1)e_d}f(x_1,\cdots, y_{d-1},y_d^{m+2})+\bar\Phi_{(m+1)e_d+e_{d-1}}f(x_1,\cdots, x_{d-1},y_{d-1},y_d^{m+2})L_{e_{d-1},\by}(\bx)\\
 & +\bar\Phi_{(m+2)e_d}f(x_1,\cdots,x_d,y_d^{m+2})L_{e_d,\by}(\bx)\\
 & = \bar\Phi_{(m+1)e_d}f(y_1,\cdots, y_{d-1},y_d^{m+2})+\sum_{j=1}^{d-1}\bar\Phi_{(m+1)e_d+e_{j}}f(x_1,\cdots, x_{j},y_{j},\cdots, y_d^{m+2})L_{e_j,\by}(\bx)\\
 & + \bar\Phi_{(m+2)e_d}f(x_1,\cdots,x_d,y_d^{m+2})L_{e_d,\by}(\bx).\end{aligned}$$
One can observe that in all of the above, we can rewrite each coefficient in $I_j$ as a sum of terms, where the first term is $\bar\Phi_{\beta}f(y_1,\cdots,y_{j-1},y_{j}^{i_j+1},\cdots,y_d^{i_d+1}) L_{\beta,\by}(\bx)$, and $\beta=(0,\cdots,0,i_j,\cdots,i_d), i_j>0, \vert \beta\vert=m+1.$ 
 Therefore, we can rewrite \cref{hyp1+} and use the previous equations as follows,
 \begin{equation}\label{hyp2+}
 \begin{aligned} f(\bx) & =P_{f,\by,m}(\bx)+\sum_{\vert\beta\vert=m+1}\bar\Phi_{\beta}f(\bar\by_{\beta})L_{\beta,\by}(\bx)+\\
 &\sum_{\substack{{\vert \beta\vert =m+2}\\\beta=(i_1,\cdots,i_d)\\i_1>0}}\bar\Phi_{\beta}f(x_1,y_1^{i_1},y_2^{i_2+1},\cdots,y_d^{i_d+1})L_{\beta,\by}(\bx)\\
 &+ \sum_{\substack{{\vert \beta\vert =m+2}\\ \beta=(0,i_2,\cdots,i_d)\\i_2>0}}\bar\Phi_{\beta}f(x_1,x_2,y_2^{i_2},y_3^{i_3+1},\cdots,y_d^{i_d+1})L_{\beta,\by}(\bx)\\
 & + \cdots+\bar\Phi_{(m+2)e_d}f(x_1,\cdots,x_d,y_d^{m+2})L_{(m+2)e_d,\by}(\bx).
 \end{aligned}\end{equation}
 Since  $P_{f,\by,m+1}(\bx)=P_{f,\by,m}(\bx)+\sum_{\vert\beta\vert=m+1} \bar\Phi_{\beta}f(\bar\by_{\beta})L_{\beta,\by}(\bx)$, we have the statement to be true for $k=m+1$. This completes the proof.
\end{proof}

We also give proof of the following lemmata whose statements is used in the paper several times. 
\begin{lemma}\label{verybasic}
Let $g:\Q_\nu\to\Q_\nu$ be a $C^{k+1}$ map. Then for any $y_1,\cdots,y_{k+1}\in\Q_\nu,$
$$\bar\Phi^k g'(y_1,\cdots, y_{k+1})= \bar\Phi^{k+1}g(y_1,y_1,y_2,\cdots,y_{k+1})+\cdots+\bar\Phi^{k+1}g(y_1,y_2,\cdots, y_{k},y_{k+1}, y_{k+1}).$$ 
\end{lemma}
\begin{proof}
    We will prove this by induction on $k$. If $k=0$, then $
    \bar\Phi^{0}g'(y)=g'(y)=\bar\Phi^1 g(y,y).$ Suppose by the induction hypothesis, we know the conclusion for any $k$, $1\leq k\leq i.$ Now we wish to show that for $k=i+1$ the conclusion of this lemma holds. 

Note that for any $(y_1,\cdots,y_{k+1})\in \nabla^{k+1}\Q_\nu$
$$
\bar\Phi^{i+1}g'(y_1,\cdots,y_{i+2})=\frac{\bar\Phi^{i}g'(y_1,y_3,\cdots,y_{i+2})-\bar\Phi^{i}g'(y_2,\cdots,y_{i+2})}{(y_1-y_2)}.
$$
    Using the induction hypothesis,
    $$\begin{aligned}
&\bar\Phi^{i}g'(y_1,y_3,\cdots,y_{i+2})\\
&=\bar\Phi^{i+1}g(y_1,y_1,y_3,\cdots,y_{i+2})+\bar\Phi^{i+1}g(y_1,y_3,y_3,y_4,\cdots,y_{i+2})\cdots+\bar\Phi^{i+1}g(y_1,y_3,\cdots, y_{i+2},y_{i+2})
    \end{aligned}
    $$ 
    
    and$$
    \begin{aligned}
&\bar\Phi^{i}g'(y_2,\cdots,y_{i+2})\\
&=\bar\Phi^{i+1}g(y_2,y_2,\cdots,y_{i+2})+\bar\Phi^{i+1}g(y_2,y_3,y_3,y_4,\cdots,y_{i+2})+\cdots+\bar\Phi^{i+1}g(y_2,\cdots,y_{i+2},y_{i+2}).
\end{aligned}
    $$
    Now substituting the above two, we get 
    $$
    \begin{aligned} &\bar\Phi^{i+1}g'(y_1,\cdots,y_{i+2})\\&=
\frac{\bar\Phi^{i+1}g(y_1,y_1,y_3,\cdots,y_{i+2})-\bar\Phi^{i+1}g(y_2,y_2,y_3,\cdots,y_{i+2})}{y_1-y_2}+\\
& \hspace{2.5 in} \bar\Phi^{i+2}g(y_1,y_2,y_3,y_3,\cdots,y_{i+2})+\cdots+\bar\Phi^{i+2}g(y_1,y_2,\cdots,y_{i+2},y_{i+2})\\
& =\frac{\bar\Phi^{i+1}g(y_1,y_1,y_3,\cdots,y_{i+2})-\bar\Phi^{i+1}g(y_1,y_2,\cdots,y_{i+2})}{y_1-y_2}+\\& \hspace{1.0 in} \frac{\bar\Phi^{i+1}g(y_1,y_2,\cdots,y_{i+2})-\bar\Phi^{i+1}g(y_2,y_2,y_3,\cdots,y_{i+2})}{y_1-y_2}+\\
& \hspace{2.5 in} \bar\Phi^{i+2}g(y_1,y_2,y_3,y_3,\cdots,y_{i+2})+\cdots+\bar\Phi^{i+2}g(y_1,y_2,\cdots,y_{i+2},y_{i+2})\\
& =\bar\Phi^{i+2}g(y_1,y_1,y_2,y_3,\cdots,y_{i+2})+\bar\Phi^{i+2}g(y_1,y_2,y_2,y_3,\cdots,y_{i+2})+\\
& \hspace{2.5 in} \bar\Phi^{i+2}g(y_1,y_2,y_3,y_3,\cdots,y_{i+2})+\cdots+\bar\Phi^{i+2}g(y_1,y_2,\cdots,y_{i+2},y_{i+2}).
    \end{aligned}
    $$
    Thus the proof is completed using the above and the definition of $\bar\Phi^{k+1}g'$.
\end{proof}

\begin{lemma}\label{apn: lemma 1}
 For $\beta=(i_1,\cdots, i_d)$ any multi index with $\vert\beta\vert=l$ and any $(\bx_1,\cdots,\bx_d)\in B(\mathbf{0},1)^{\beta(1)}$, for any  $f:\Q_\nu^d\to \Q_\nu,$ that is $C^{l+1},$ $1\leq j\leq d$, we get 
\begin{equation}\label{claim0}\begin{aligned}
& i_j\bar \Phi_\beta f(\bx_1,\cdots,\bx_{d})\\
& = \bar\Phi_{\beta-e_j}\partial_{e_j}f(\bx_1,\cdots,\bx_{j-1},x_{j,1},\cdots,x_{j,i_j},\bx_{j+1},\cdots,\bx_d)\\
& ~~~~~ +\sum_{i=1}^{i_j} \bar\Phi_{\beta+e_j}f(\bx_1,\cdots,\bx_{j-1},x_{j,1},\cdots,x_{j,i-1},x_{j,i},x_{j,i},x_{j,i+1},\cdots,x_{j,i_j+1},\bx_{j+1},\cdots,\bx_d)(x_{j,i_j+1}-x_{j,i}).
\end{aligned}
\end{equation}
\end{lemma}
\begin{proof}
    First note that by \cref{verybasic}, $$\begin{aligned} &\bar\Phi_{\beta-e_j}\partial_{e_j}f(\bx_1,\cdots,\bx_{j-1},x_{j,1},\cdots,x_{j,i_j},\bx_{j+1},\cdots,\bx_d)\\& 
    = \bar\Phi_{\beta}f(\bx_1,\cdots,\bx_{j-1},x_{j,1}, x_{j,1},x_{j,2},\cdots, x_{j,i_j},\bx_{j+1},\cdots,\bx_d)+\\
    & \cdots+ \bar\Phi_{\beta}f(\bx_1,\cdots,\bx_{j-1},x_{j,1},x_{j,2},\cdots,x_{j,i_j-1},x_{j,i_j},x_{j,i_j},\bx_{j+1},\cdots,\bx_d).
    \end{aligned}$$
    Also, note there are $i_j$ many summands in the above sum. 
    
    Then by definition, for any $1\leq i\leq i_j,$ \begin{equation*}
        \begin{aligned}
          & \bar \Phi_\beta f(\bx_1,\cdots,\bx_{d})-  \bar\Phi_{\beta}f(\bx_1,\cdots,\bx_{j-1}, x_{j,1},\cdots,x_{j,i-1}, x_{j,i}, x_{j,i},x_{j,i+1}, \cdots, x_{j,i_j},\bx_{j+1},\cdots,\bx_d)  \\
         = & \bar \Phi_{\beta+e_j} f(\bx_1,\cdots,\bx_{j-1}, x_{j,1},\cdots,x_{j,i-1}, x_{j,i}, x_{j,i},x_{j,i+1}, \cdots, x_{j,i_j}, x_{j, i_j+1},\bx_{j+1},\cdots,\bx_d) (x_{j,i_j+1}-x_{j,i}).
        \end{aligned}
    \end{equation*}
    Combining these two observations, the lemma follows.
\end{proof}

The following is a version of \cite[Lemma 29.2, (v)]{Sc}. We provide a proof to make this article self contained. 
\begin{lemma}\label{chain chain} Let $g:\Q_\nu^d\to\Q_\nu$ be $C^{l+1}$ map. For any $\alpha$ multi-index with $\vert\alpha\vert=l$, we have the following, for any $z\in\Q_\nu, \bx_{\alpha(1)}\in \Q_\nu^{\vert \alpha\vert+d},$ $1\leq i\leq d,$
\begin{equation}\label{basic2}
\begin{aligned}
    &\bar\Phi_{\alpha+e_i}\left( L_{e_i,\by}  g \right)(z,\bx^{(i)}_{\alpha(1)})\\
  &= L_{e_i,\by}(z)\bar\Phi_{\alpha+e_i}g(z,\bx^{(i)}_{\alpha(1)})+ \bar\Phi_{\alpha}g (\bx_{\alpha(1)}).
  \end{aligned}\end{equation}
\end{lemma} 
\begin{proof}
The proof is done by induction on $l$. When $l=0,$ then $\alpha=0$,
Suppose $(z,\bx)\in\Q_\nu^{d+1},$ $z\neq x_i$ and $1\leq i\leq d$. Then $$\begin{aligned}
&\Phi_{e_i}(L_{e_i,\by}g)(x_1,\cdots,x_{i-1},z,x_i, x_{i+1}, \cdots,x_d)\\
& =\frac{(z-y_i) g(x_1,\cdots,x_{i-1}, z, x_{i+1},\cdots,x_d)-(x_i-y_i)g(\bx)}{z-x_i}\\
&= (z-y_i) \Phi_{e_i}g(z,\bx^{(i)})+g(\bx).
\end{aligned}$$
Then taking limits gives us \cref{basic2} for $l=0.$
Now suppose by the induction hypothesis, that \cref{basic2} is true for when $\vert \alpha\vert\leq l-1.$ Now let us take $\alpha$ such that $\vert\alpha\vert=l$, and fix $i$, where $1\leq i\leq d.$ There exists some $1\leq j\leq d$ such that $\alpha=\beta+e_j$, where $\vert \beta\vert=l-1.$
Then for $z\in\Q_\nu, \bx_{\alpha(1)}\in \Q_\nu^{\vert \alpha\vert+d}$, $x_{j,1}\neq x_{j,2}$, 
$$\begin{aligned}
&\Phi_{\alpha+e_i} \left( L_{e_i,\by}  g \right)(z,\bx^{(i)}_{\alpha(1)})\\
& = \frac{\Phi_{\beta+e_i}( L_{e_i,\by}  g)(z, {\hat\bx_{\alpha(1),j,2}^{(i)}})- \Phi_{\beta+e_i}( L_{e_i,\by}  g)(z, \hat\bx^{(i)}_{\alpha(1),j,1})}{x_{j,1}-x_{j,2}},  
\end{aligned}
$$
where $\hat\bx_{\alpha(1),j,1}=(\bx_1,\bx_{j-1}, x_{j,2}, \cdots, x_{j,i_j+1},\bx_{j+1},\cdots,\bx_d)$,  and similarly for  $\hat\bx_{\alpha(1),j,2}$ with $x_{j,2}$ missing from $\bx_{\alpha(1)}$ among the $j$-th corodinates.
Now using the induction hypothesis, we get 
$$  \Phi_{\beta+e_i}( L_{e_i,\by}  g)(z, \hat\bx^{(i)}_{\alpha(1),j,2})=  L_{e_i,\by}(z)\Phi_{\beta+e_i} g(z,\hat\bx^{(i)}_{\alpha(1),j,2})+ \Phi_{\beta}g (\hat\bx_{\alpha(1),j,2}).$$
and 
$$  \Phi_{\beta+e_i}( L_{e_i,\by}  g)(z, \hat\bx^{(i)}_{\alpha(1),j,1})=  L_{e_i,\by}(z)\Phi_{\beta+e_i}g(z,\hat\bx^{(i)}_{\alpha(1), j, 1})+ \Phi_{\beta}g (\hat\bx_{\alpha(1),j,1}).$$
Using the above two, and $\alpha=\beta+e_j,$ 
we get $$\begin{aligned}
\Phi_{\alpha+e_i} \left( L_{e_i,\by}  g \right)(z,\bx_{\alpha(1)})
&= \frac{(z-y_i)\Phi_{\beta+e_i} g(z,\hat\bx^{(i)}_{\alpha(1),j,2})- (z-y_i)\Phi_{\beta+e_i} g(z,\hat\bx^{(i)}_{\alpha(1),j,1}) }{x_{j,1}-x_{j,2}}
+ \Phi_{\beta+e_j} g (\bx_{\alpha(1)})\\
&=(z-y_i) \Phi_{\alpha+e_i}g(z,\bx^{(i)}_{\alpha(1)})+ \Phi_{\alpha}g(\bx_{\alpha(1)}).
\end{aligned}
$$
\end{proof}

\begin{lemma}\label{bi}Suppose $\f=(f_1,\cdots,f_n):U\subset \Q_\nu^d\to\Q_\nu^n$ is a nonsingular $C^{l+1}$ map at $\bx_0$, we can find a neighbourhood $V$ of $\bx_0$ such that $\f|_{V}$ is bi-Lipschitz.
\end{lemma}
\begin{proof}
    For simplicity of notation and to make it more readable, we give a proof for $d=2$. The proof for the other cases is exactly the same. Note that, for $\bx=(x_1,x_2), \by=(y_1,y_2)$, for every $1\leq i\leq n$, $$\begin{aligned}
    f_i(\bx)-f_i(\by)=& f_i(x_1,x_2)-f_i(x_1,y_2)+f_i(x_1,y_2)-f_i(y_1,y_2)\\
    = & \bar\Phi_{e_2}f_i(x_1,x_2,y_2)(x_2-y_2)+ \bar\Phi_{e_1}f_i(x_1,y_1,y_2)(x_1-y_1).
    \end{aligned}
    $$

This implies that 
\begin{equation}\label{Axy}
\f(\bx)-\f(\by)= A_{\bx,\by}(\bx-\by)^T,
\end{equation}
where $$A_{\bx,\by}= \begin{bmatrix}
    \bar\Phi_{e_1}f_1(x_1,y_1,y_2) & \bar\Phi_{e_2} f_1(x_1,x_2,y_2)\\
    \vdots & \vdots\\
    \bar\Phi_{e_1}f_n(x_1,y_1,y_2) & \bar\Phi_{e_2} f_n(x_1,x_2,y_2)
\end{bmatrix}.$$
Let us also denote by $A_{\bx,\by}(r,s)$ the $2\times 2$ matrix with $r,s$-th rows from the above matrix $A_{\bx,\by}$. Note that 
$A_{\bx_0,\bx_0}=\nabla \f(\bx_0)$, by \eqref{beta}.  
Since $\f$ is nonsingular at $\bx_0$, we know $\nabla \f(\bx_0)$ has rank $2$.  This implies that there exist $1\leq i_1<i_2\leq n $ such that $$A_{\bx_0,\bx_0}({i_1,i_2})=\begin{bmatrix}
    \partial_{e_1}f_{i_1}(\bx_0) & \partial_{e_2}f_{i_1}(\bx_0)\\
    \partial_{e_1}f_{i_2}(\bx_0) & \partial_{e_2}f_{i_2}(\bx_0)
\end{bmatrix} $$ is invertible. By continuity of the determinant function, there exists a neighborhood ball $V$ (which is both open and closed) of $\bx_0$, such that for any $\bx,\by\in V,$ we have that $A_{\bx,\by}({i_1,i_2})$ is also invertible. Moreover, using the ultrametric norm, one can also guarantee that $\vert \det(A_{\bx,\by}(i_1,i_2))\vert $ is nonzero and constant on $V\times V.$ Thus $(\bx, \by) \to A^{-1}_{\bx,\by}(i_1,i_2)$ is a continuous map defined on $V\times V$. Hence \begin{equation}\label{end_1}\Vert \bx-\by\Vert \leq \Vert A^{-1}_{\bx,\by}({i_1,i_2})\Vert \Vert A_{\bx,\by}({i_1,i_2})(\bx-\by)^T\Vert.\end{equation} Also \begin{equation}\label{end_2}\Vert A_{\bx,\by}(i_1,i_2)(\bx-\by)^T\Vert=\Vert (f_{i_1},f_{i_2})(\bx)-(f_{i_1},f_{i_2})(\by)\Vert\leq \Vert\f(\bx)-\f(\by)\Vert. \end{equation} Since $(\bx,\by)\to \Vert A^{-1}_{\bx,\by}({i_1,i_2})\Vert$ is continuous on $V\times V$, and $V$ is a compact ball, there is a $K_1>0$ such that $\Vert A^{-1}_{\bx,\by}({i_1,i_2})\Vert \leq K_1$. This  together with \eqref{end_1} and \eqref{end_2} completes the proof of the lower bound in the Lipschitz condition. The upper bound follows from \eqref{Axy} using the compactness of $V$ and continuity of the map $(\bx,\by)\to \Vert A_{\bx,\by}\Vert.$

\end{proof}

\subsection*{Acknowledgements} The authors thank the anonymous referees for their valuable comments that helped to clarify many details of the first version of this paper. SD thanks Ralf Spatzier  and Subhajit Jana for many helpful discussions and several helpful remarks which have improved the paper.
\bibliographystyle{abbrv}
\bibliography{KT.bib}

\end{document}